\documentclass[11pt,leqno]{article}
\usepackage{hyperref}
\usepackage{amsmath}
\usepackage{amssymb}
\usepackage{graphicx}
\usepackage{color}
\newtheorem{theorem}{Theorem}[section]

\newtheorem{proposition}[theorem]{Proposition}
\newtheorem{definition}[theorem]{Definition}
\newtheorem{remark}[theorem]{Remark}

\def\square{\hbox{\vrule\vbox{\hrule\phantom{o}\hrule}\vrule}}
\def\cqfd{\hfill\square}
\newenvironment{proof}%
{\noindent{\sl Proof:}}{{\null\cqfd\medskip}}

\def\C{{\mathbb C}}

\def\N{{\mathbb N}} 
\def\R{{\mathbb R}} 

\def\Z{{\mathbb Z}}

\def\CC{\mathcal {C}}
\def\CD{\mathcal {D}}
\def\CE{\mathcal {E}}

\def\CM{\mathcal {M}}

\def\CO{\mathcal {O}}
\def\CP{\mathcal {P}}

\def\CT{\mathcal {T}}

\def\re{\mathop{\rm Re}\nolimits}
\def\im{\mathop{\rm Im}\nolimits}

\def\Hol{\mathop{\rm Hol}\nolimits}

\newcommand{\tr}{\operatorname{tr}}

\def\<{\langle}
\def\>{\rangle}

\def\ds{\displaystyle}

\makeatletter
 \@addtoreset{equation}{section}
 \makeatother
 
\title{$\CP\CT$-symmetry and Schr\"odinger operators.\\
 The double well case}
\author{
Nawal Mecherout \\
{\small Universit\'{e} de Mostaganem }\\
{\small Facult\'{e} des science exactes et informatique }\\
{\small 27000-Mostaganem, Alg\'erie}\\
{\small mecheroutnawel@yahoo.fr }
\and
Naima Boussekkine \\
{\small Universit\'{e} de Mostaganem }\\
{\small Facult\'{e} des science exactes et informatique }\\
{\small 27000-Mostaganem, Alg\'erie}\\
{\small nboussekkine@yahoo.fr}
\and
Thierry Ramond \footnote{Supported by the ANR project NOSEVOL ANR 2011 BS 010119 01.}\\
{\small LMO (UMR CNRS 8628)}\\
{\small Universit\'e Paris Sud}\\ {\small FR 91405 Orsay, France}\\
{\small thierry.ramond@math.u-psud.fr}
\and
Johannes Sj¨\"ostrand \footnotemark[\value{footnote}] \\
{\small IMB, Universit\'e de Bourgogne (UMR CNRS 5584)}\\
{\small 9, Avenue Alain Savary, BP 47870}\\
{\small FR-21078 Dijon c\'edex}\\
{\small jo7567sj@u-bourgogne.fr }
}
\begin{document}

\maketitle

\newpage

\begin{abstract}
  We study a class of $\CP\CT$-symmetric semiclassical Schr\"{o}dinger
  operators, which are perturbations of a selfadjoint one.  Here, we
  treat the case where the unperturbed operator has a double-well
  potential. In the simple well case, two
  of the authors have proved in \cite{BoMe14} that, when the potential is analytic, the eigenvalues stay
  real for a perturbation of size $\CO(1)$. We show here, in the
  double-well case, that the eigenvalues stay real only for
  exponentially small perturbations, then bifurcate into the complex
  domain when the perturbation increases and we get precise asymptotic
  expansions. The proof uses complex WKB-analysis, leading to a fairly
  explicit quantization condition.
\end{abstract}

\tableofcontents

\newpage
\section{Introduction and results}
\label{int}

Operators that are $\CP\CT$-symmetric have been proposed in quantum
mechanics as an alternative to selfadjoint ones. From a physicist
point of view, it is of course very important to verify that the
spectrum of such operators is real and there has been a considerable
activity in this area \cite{Be05}, \cite{BeBeMa02}, \cite{BeMa10}, \cite{BeBoMe98}, \cite{CaGr05},
\cite{CaGrCa06}, \cite{CaGrSj05}, \cite{CaGrSj07}, \cite{LeZn00}, \cite{Mo05}, \cite{Mo05a}, \cite{Mo02b}, \cite{Mo02c}, \cite{Sj11},
\cite{Si11}, \cite{ZnCaBaRo00}\dots  There is in particular an issue \cite{BeFrGuJo12} of the Journal of Physics A 
which is devoted to non-selfadjoint operators in quantum physics, where the $\CP\CT$-symmetry property plays the main role. 
In a recent paper  \cite{CaGr14}, E. Caliceti and S. Graffi  study qualitatively if, in a perturbative setting, the phenomenon called $\CP\CT$-symmetry phase transition occurs for $\CP\CT$-symmetric polynomial potentials, that is if the eigenvalues bifurcate from real to complex values. The reader may also find interesting references in that paper.

It has been proved recently  by two of the authors in \cite{BoMe14} in the analytic category, that the $\CP\CT$-symmetric perturbation of a semiclassical Schr\"odinger operator with a real-valued single well potential,  have real spectrum  even for perturbations of size $\CO(1)$. Here we address the double-well case, where numerical results suggest that the situation is very different. For example, in Figure \ref{int.fig.1}, we have plotted the eigenvalues of the Schr\"odinger operator
$P_{\varepsilon,h}=h^2D^2 +0.05x^4-.5x^2+i\varepsilon x$ in a neighborhood of the barrier top, for $h=.01$ and $\varepsilon=k.10^{-m}$, $k=1,2,3,4$ and $m=2,3,4,5$. They have been computed by a simple finite difference scheme using the package scipy in python, and plotted using the package matplotlib \cite{Hu07}. It appears that the eigenvalues close to $E_{0}>0$  are real, giving a numerical illustration of the results in \cite{BoMe14} since we are then in a simple well situation. In the double well case, that is for  $E_{0}<0$, the eigenvalues close to $E_{0}$ seem to be non-real except for extremely small $\varepsilon$, and this is the phenomenon we consider in this paper.

\begin{figure}[h]
\begin{center}
\includegraphics[width=13cm]{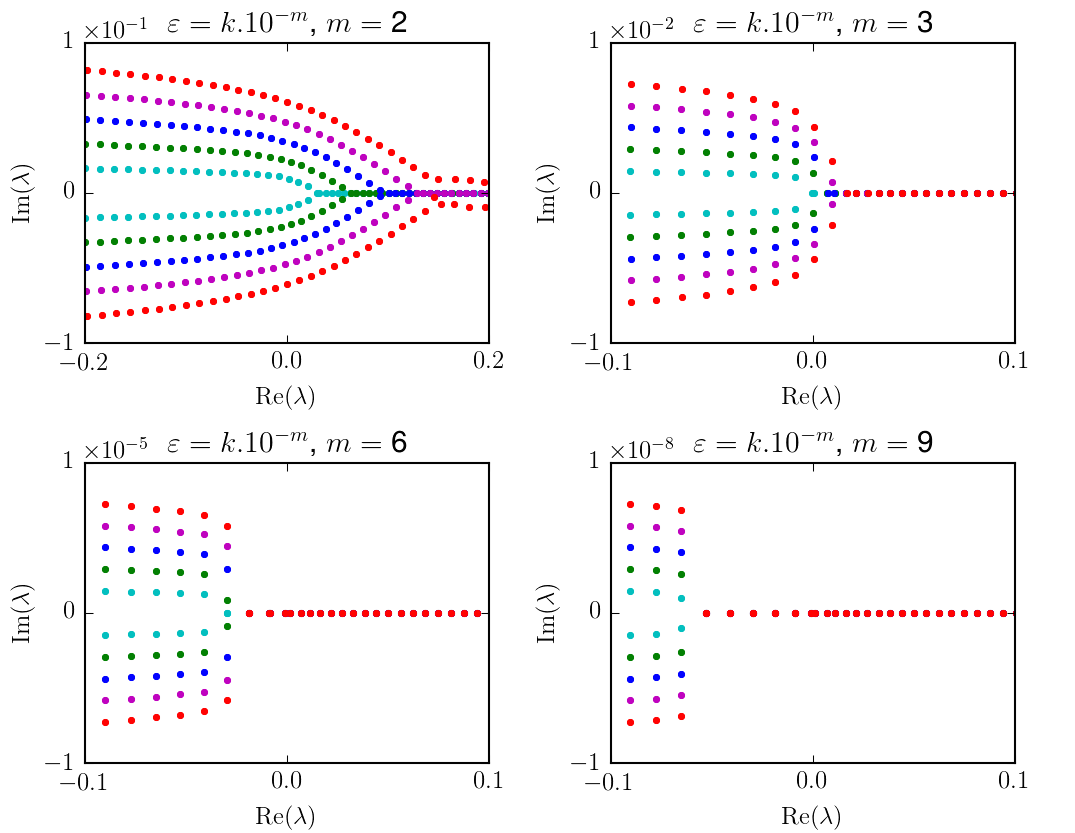}
\caption{Eigenvalues of $P_{\varepsilon,h}=h^2D^2 +0.05x^4-.5x^2+i\varepsilon x$ near 0, with $h=10^{-2}$ and $\varepsilon=k.10^{-m}, k=1,\dots,5$. \label{int.fig.1}}
\end{center}
\end{figure}

We recall that an operator is said to be $\CP\CT$-symmetric when it commutes with the operator $\CP\CT$, where $\CP$ is the parity operator, and $\CT$ the time-reversal operator given by
\begin{equation}
\CP u(x)=u(-x) \mbox{ and } \CT u(x)= \overline{u( \overline x)}.
\label{int.1}
\end{equation}
Here, we study small perturbations $P_{h,\varepsilon}=P_{\varepsilon}(x,hD)$ of self-adjoint semiclassical Schr¨\"odinger operators  of the form\begin{equation}
P_{\varepsilon}(x,hD)=-h^2\frac{d^2}{dx^2}+V_{\varepsilon}(x), \quad V_{\varepsilon}(x)=V_{0}(x)+i\varepsilon W(x),
\end{equation}
where $V_{0}$ and $W$ are smooth functions on $\R$ and $E_{0}\in \R$ a fixed energy, satisfying the following assumptions:

\begin{description}

\item[(A1)] $V_{0}$ is $\CC^\infty$ and real-valued on $\R$.

\item[(A2)] There exists $m_{0}>0$ such that, with $\< x\>= (1+x^2)^{\frac{1}{2}}$,
$$
\forall k\in \N,\exists C_{k}>0,\; \forall x\in \R,\;  \vert V^{(k)}_{0}(x)\vert\leq C_{k}\< x\>^{m_0-k}.
$$
and
$$
\forall x\in \R\setminus ]-C_{0},C_{0}[,\;  \vert V_{0}(x)-E_{0}\vert\geq
\frac{1}{C_{0}}\<x\>^{m_{0}}.
$$
\item[(A3)] For some $E_{0}\in \mathbb{R}$, the equation 
$V_{0}(x)=E_{0}$ has exactly four solutions $\alpha_{\ell}<\beta_{\ell}<\beta_{r}<\alpha_{r}$, with $V_{0}'(\alpha_{\ell})<0$, $V_{0}'(\beta_{\ell})>0$, $V_{0}'(\beta_{r})<0$, and $V_{0}'(\alpha_{r})>0$.

\end{description}

Therefore, the operator $P_{h,0}$ on $L^2(\R)$ with domain 
$$
\CD=\{u\in H^2(\R),\; \<x\>^{m_{0}}u\in L^2(\R)\},
$$ 
is a self-adjoint Schr\"odinger operator with a double-well potential, and its spectrum consists only in real eigenvalues with multiplicity 1.

\begin{figure}[!h]
\begin{center}
\def\svgwidth{8cm}
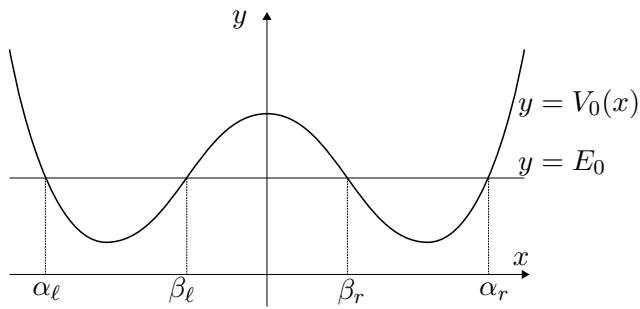
\caption{The unperturbed potential $V_{0}$\label{int.fig.2}}
\end{center}
\end{figure}

Concerning the perturbation $i\varepsilon W$, we suppose that, with $m_{0}$ given in (A2),

\begin{description}
\item [(A4)] $W$ is a $\CC^\infty$, real-valued function on $\R$, such that
$$
\forall k\in \N,\exists C_{k}>0,\; \forall x\in \R,\;  \vert W^{(k)}(x)\vert\leq C_{k}\< x\>^{m_0-k}.
$$
\end{description}

Then, for any $\varepsilon\in ]-\varepsilon_{0},\varepsilon_{0}[$ with
$\varepsilon_{0}>0$ small enough, the spectrum of the unbounded,
closed operator $P_{h,\varepsilon}$ on $L^2(\R)$ with domain $\CD$, is
still discrete in a complex ($h$-independent) neighborhood of $E_{0}$.
We are interested in the semiclassical asymptotics (i.e. as $h\to 0$)
of the eigenvalues of $P_{h,\varepsilon}$ near $E_{0}$, and we want to
know whether they stay real or not for $\varepsilon>0$ 
small, under the following $\CP\CT$-symmetry assumption:

\begin{description}
\item[(A5)] The operator $P_{h,\varepsilon}$ is $\CP\CT$-symmetric, that is
$$
[P_{h,\varepsilon},\CP\CT]=0,
$$
where $\CP$ is the parity operator, and $\CT$ the time-reversal operator given in (\ref{int.1}).
\end{description}

In terms of the real and imaginary part of the potential $V_{\varepsilon}$, this assumption is equivalent to the property that $V_{0}$ is even and $W$ is odd. In particular, when $\varepsilon=0$ we have a symmetric double-well potential, and the turning points satisfy
\begin{equation}
\alpha_{\ell}=-\alpha_{r},\; \beta_{\ell}=-\beta_{r}.
\end{equation}

Since we are going to use complex WKB constructions in a neighborhood of the wells, we need to suppose  also that

\begin{description}
\item[(A6)] $V_{0}$ and $W$ have analytic extensions to a neighboorhood 
$U$ in $\mathbb{C}$ of the convex hull of $\{x\in \R ;\; V_{0}(x)\leq E_{0}\}.$

\end{description}

\begin{remark}\label{int.2}\sl In (A2), and the corresponding assumption on
  $W$, we can also treat the case $m_{0}=0$.
Then  $P_{h,0}$ is still self-adjoint, and its spectrum in $]-\infty,E_{0}+1/C[$ consists also only in eigenvalues.
Our results remain valid in this case.
\end{remark}

In this paper, we will first obtain a Bohr-Sommerfeld like quantization
condition for the eigenvalues of $P_{h,\varepsilon}$ in $D(E_{0},1/C)$. This result does
not rely on the $\CP\CT$-symmetry assumption on $P_{h,\varepsilon}$.

Let us define the relevant action integrals.  By the implicit function
theorem, there exists $\varepsilon_{0}>0$ such that the equation
$V_{\varepsilon}(z)-E=0$ has exactly four solutions for
$(E,\varepsilon)\in D(E_{0},\varepsilon_{0})\times
D(0,\varepsilon_{0})$, that we denote $\alpha_{\ell}(E,\varepsilon)$,
$\beta_{\ell}(E,\varepsilon)$, $\beta_{r}(E,\varepsilon)$ and
$\alpha_{r}(E,\varepsilon)$, depending analytically on $E,\varepsilon $,
in such a way that,
\begin{equation}
\alpha_{\bullet}(E_{0},0)=\alpha_{\bullet},\quad \beta_{\bullet}(E_{0},0)=\beta_{\bullet},\quad
\bullet =\ell,r.
\end{equation}
These analytic functions of $(E,\varepsilon)$ are called turning points at energy $E$. The action integrals are the functions given by 
\begin{eqnarray}
I_{\bullet}(E,\varepsilon) =\int_{\alpha_{\bullet}(E,\varepsilon)}^{\beta_{\bullet}(E,\varepsilon)} (E-V_{\varepsilon}(t))^{\frac{1}{2}} dt,
\label{int.3}
\end{eqnarray}
where $\bullet=\ell,r$ and the branch of $t\mapsto(E-V_{\varepsilon}(t))^{\frac{1}{2}}$ is
the one that is real and positive for $(E,\varepsilon )=(E_0,0)$ on the segment
between $\alpha_{\bullet}$ and $\beta_{\bullet}$, $\bullet=\ell,r$,
and
\begin{equation}
J(E,\varepsilon) = \int_{\beta_{l}(E,\varepsilon)}^{\beta_{r}(E,\varepsilon)} (V_{\varepsilon}(t)-E)^{\frac{1}{2}} dt,
\label{int.4}
\end{equation}
where $t\mapsto(V_{\varepsilon}(t)-E)^{\frac{1}{2}}$ is real and
positive for $(E,\varepsilon)=(E_0,0)$ on the segment between
$\beta_{\ell}$ and $\beta_{r}$. These functions $I_\ell$, $I_r$ and $J$ are analytic
with respect to $E,\,\varepsilon $.

We shall say that a function $f(z,h)$ defined in $\Omega\times ]0,h_{0}]$, analytic with respect to $z\in\Omega$, a domain of $\C^d$, has an asymptotic expansion
$$
f(z,h)\sim \sum_{j\geq 0}f_{j}(z)h^j
$$
in $\Hol(\Omega)$, when the $f_{j}$'s are holomorphic in $\Omega$ and,
for all
compact subsets
 $K\subset \Omega$ and all $N\in \N$,
\begin{equation}\label{int.5}
\Vert f(z,h)-\sum_{j = 0}^N f_{j}(z)h^j\Vert_{L^\infty(K)}=\CO(h^{N+1}).
\end{equation}

For a function $k(E,\varepsilon)$ we shall denote, for $\varepsilon\in \R$,
\begin{equation}
k^*(E,\varepsilon)=\overline{k(\overline E,-\varepsilon)}
\quad \mbox { and }\quad 
k^\dagger (E,\varepsilon )=\overline{k(\overline{E},\varepsilon )}.
\end{equation}

\begin{theorem}
\label{int.6}
\sl Assume (A1)--(A4) and (A6). There exists a fixed neighborhood
$\Omega=\Omega _1\times \Omega _2 $ of $(E_0,0)$ in $\mathbb{C}\times \mathbb{C}$ and
holomorphic functions $\gamma _{\pm}^\ell(E,\varepsilon ;h)$, $\gamma
_{\pm}^r(E,\varepsilon ;h)$ with complete asymptotic expansions in
$\mathrm{Hol\,}(\Omega )$ and of the form $1+{\cal O}(h)$ such that if
\begin{equation}\label{int.7}
\begin{split}
f(E,\varepsilon ;h )=&\frac{1}{4}\left(e^{iI_\ell /h}\gamma _-^\ell +e^{-iI_{\ell} /h}\gamma
  _+^\ell \right)
\left(e^{iI_{r} /h}\gamma _+^r +e^{-iI_{r} /h}\gamma _-^r \right)\\
&-\frac{1}{4}e^{-2J/h}\sin \left(I_{\ell} /h \right) \sin\left(I_{r}/h
\right) ,
\end{split}
\end{equation}
then when $h>0$ is small enough and $\varepsilon \in \Omega _1\cap \mathbb{R}$, the
eigenvalues of $P_\varepsilon $ in $\Omega _1$ coincide with the zeros of
$f(\cdot ,\varepsilon ;h) $ in the same set. The algebraic multiplicity
of each such eigenvalue $E$ coincides with the multiplicity of $E$ as a
zero of $f(\cdot ,\varepsilon ;h)$. 

\par
 We may assume that $\Omega $ is invariant
under the map $(E,\varepsilon )\mapsto (\overline{E},\varepsilon )$, and we have 
\begin{equation}\label{int.8}
  I_\ell^*=I_\ell,\ I_r^*=I_r,\  J^*=J,\ \left (\gamma _-^\ell\right)^*=\gamma
  _+^\ell,\ \left(\gamma _-^r\right)^*=\gamma _+^r. 
\end{equation}

\par Assuming also (A5), we have
\begin{equation}\label{int.9}
\begin{split}
  f(E,\varepsilon )=
\rho \rho ^\dagger \cos (\widetilde{I}/h)\cos (\widetilde{I}^\dagger /h)
-\frac{1}{4}e^{-2J/h}\sin (I/h)\sin (I^\dagger /h) 
,
\end{split}
\end{equation}
where $\rho (E,\varepsilon ;h)$, $\widetilde{I}(E,\varepsilon ;h)$ have
complete asymptotic expansions in $\mathrm{Hol\,}(\Omega )$ with $\rho
=1+{\cal O}(h)$, $\widetilde{I}=I+{\cal O}(h^2)$. Further, $\widetilde{I}^*=\widetilde{I}$, $\rho
^*=\rho $ and $J^*=J=J^\dagger$.
\end{theorem}

\begin{remark}\sl The function  $f$ in (\ref{int.7}) can also be written
\begin{equation}\label{int.10}
f(E,\varepsilon ;h)=\rho_{\ell} \rho_{r}\cos (\widetilde{I}_\ell /h)\cos
(\widetilde{I}_r/h) -\frac{1}{4}e^{-2J/h}\sin (I_\ell /h)\sin (I_r /h),
\end{equation}
where 
$$
\widetilde{I}_{\ell} -I_{\ell},\ \widetilde{I}_{r} -I_{r} ={\cal O}(h^2),\ \
\rho_{\ell} -1,\ \rho_{r}-1={\cal O}(h).
$$
These quantities have complete asymptotic expansions in powers of
$h$ and enjoy the symmetries,
$$
\rho _\ell ^*=\rho _\ell,\ \rho _r ^*=\rho _r, \ \
\widetilde{I}_\ell^*=\widetilde{I}_\ell,\ \widetilde{I}_r^*=\widetilde{I}_r.
$$

Under the additional assumption (A5), we have $\rho _r=\rho _\ell ^*$,
$\widetilde{I}_r=\widetilde{I}_\ell ^*$ and we get (\ref{int.9}) with
$\rho =\rho_\ell$, $I=I_\ell$, $\widetilde{I}=\widetilde{I}_\ell$.
\end{remark}

\par The zeros of $E\mapsto \cos (\widetilde{I}_\ell /h)$ and
$E\mapsto \cos (\widetilde{I}_r /h)$ can be thought of as ``approximate
eigenvalues for the left and right potential wells
respectively''. They are given by the Bohr-Sommerfeld quantization
conditions,
\begin{equation}\label{int.11}
2\widetilde{I}_\ell (E,\varepsilon ;h)=(2k+1)\pi h,\ k\in \mathbb{Z},
\end{equation}
and
\begin{equation}\label{int.12}
2\widetilde{I}_r (E,\varepsilon ;h)=(2k+1)\pi h,\ k\in \mathbb{Z},
\end{equation}
respectively. The last term in (\ref{int.10}) is exponentially small
and represents the tunneling interaction between the potential wells.
The solutions
$E=\widetilde{E}_k^\bullet (\varepsilon ;h)$ of (\ref{int.11}), (\ref{int.12})
are situated on the curves $\widetilde{\Gamma }_\bullet (\varepsilon ;h)$
 defined by 
\begin{equation}\label{int.13}
\im\widetilde{I}_\bullet (E,\varepsilon ;h)=0,
\end{equation}
respectively, and the distance between consecutive zeros
$\widetilde{E}^\bullet _k$ and $\widetilde{E}^\bullet _{k+1}$ is of the
order $h$. Here, we use that, for $\bullet=\ell,r$,
$$
\partial _EI_\bullet =\frac{1}{2}\int_{\alpha _\bullet }^{\beta _\bullet
}(E-V_\varepsilon (x))^{-\frac{1}{2}} dx\ne 0.
$$
When $\varepsilon =0$,  $\widetilde{\Gamma
}_\bullet$ are real neighborhoods of $E_0$ and the $\widetilde{E}^\bullet_k$
 are real. We notice that
$$
\re \widetilde{E}^\ell_k(\varepsilon )=E_k^\ell (0)+{\cal O}(\varepsilon ^2).
$$

It is clear that the set of zeros of $f(\cdot ,\varepsilon ;h)$ is (in a
suitable sense) exponentially close to the union of the solutions of
(\ref{int.11}) and (\ref{int.12}). Below we give a such a detailed
result in the ${\cal PT}$-symmetric case. 

Now we adopt the assumptions (A1)--(A6) as well as the following
assumption:
\begin{description}
\item[(A7)] We have
$$
\int_{\alpha _\ell}^{\beta _\ell}(E_0-V_0(x))^{-\frac{1}{2}}W(x)dx\ne 0.
$$
\end{description}
Notice here that 
$$
2i\partial _\varepsilon I_{\ell} (E,\varepsilon )=\int_{\alpha _\ell}^{\beta
  _\ell}(E-V_\varepsilon (x))^{-\frac{1}{2}}W(x)dx.
$$
Possibly after the substitution $(\varepsilon ,W)\mapsto (-\varepsilon
,-W)$, which does not change $P_\varepsilon $, we may assume that the
integral in (A7) is $>0$.

\par Under the assumption (A7), we have $\tilde I_{\ell}^\dagger =\tilde I_{r}$ and hence for
real $\varepsilon $, that  $\widetilde{E}_k^r=\overline{\widetilde{E}_k^\ell}$.


\begin{theorem}\label{int.14}\sl We make the assumptions (A1) to
  (A7). The values $\widetilde{E}_k:=\widetilde{E}_k^\ell$ and
  $\widetilde{E}_k^r=\overline{\widetilde{E}_k}$ are situated on
  the curves $\widetilde{\Gamma }=\widetilde{\Gamma }_\ell$ and
  $\overline{\widetilde{\Gamma }}=\widetilde{\Gamma }_r$, where
  $\widetilde{\Gamma }$ is of the form
\begin{equation}\label{int.15}
\im E=\widetilde{g}(\re E,\varepsilon ;h).
\end{equation}
Here
\begin{equation}\label{int.16}
\widetilde{g}(t,\varepsilon ;h)\sim g(t,\varepsilon)+hg_1(t,\varepsilon
)+...\hbox{ in }\mathrm{Hol\,}(\mathrm{neigh\,}(E_0,0),\mathbb{C}^2)
\end{equation}
is real for $(t,\varepsilon )$ real, and we have,
\begin{equation}\label{int.17}
\widetilde{g}(t,\varepsilon )=\left(\frac{i\partial _\varepsilon I }{\partial
    _EI}(t,0)+{\cal O}(h^2) \right) \varepsilon +{\cal O}(\varepsilon ^2).
\end{equation}

\par There exists a fixed neighborhood $\Omega =\Omega _1\times \Omega
_2$ of $(E_0,0)\in \mathbb{C}\times \mathbb{R}$, such that for
$\varepsilon \in \Omega _2$ and for $h>0$ small enough,
\begin{equation}\label{int.18}
f^{-1}(0,\varepsilon )\subset \bigcup_k
D\left( \widetilde{E}_k,r(\widetilde{E}_k,\varepsilon )\right) \cup
\bigcup_{k}D\left(\overline{\widetilde{E}_k},r\left(\overline{\widetilde{E}_k},\varepsilon
  \right) \right),
\end{equation}
where $r(E,\varepsilon )\le Ch e^{-\re J(E,\varepsilon )/h}$ is given by
\begin{equation}\label{int.19}
r(E,\varepsilon )=Ch \min \left( 1,\max (h/\varepsilon ,1)e^{-\re
    J(E,\varepsilon )/h} \right) e^{-\re J(E,\varepsilon )/h}.
\end{equation}
and $C>0$ is large enough.
Moreover,
\begin{itemize}
\item when these discs are disjoint, $f(\cdot ,\varepsilon )$ has precisely one zero in each of
  $D\left(\widetilde{E}_k,r\left(\widetilde{E}_k,\varepsilon \right)\right)$ and\- $D\left(\overline{\widetilde{E}_k},r\left(\overline{\widetilde{E}_k},\varepsilon
  \right) \right)$.
\item in general $f(\cdot ,\varepsilon )$ has precisely 2 zeros in
  $$D\left(\widetilde{E}_k,r\left(\widetilde{E}_k,\varepsilon
\right)\right)\cup
    D\left(\overline{\widetilde{E}_k},r\left(\overline{\widetilde{E}_k},\varepsilon
      \right) \right).$$
\end{itemize}
\end{theorem}

\par
We finally discuss the more precise behaviour of the eigenvalues when
$|\varepsilon |$ is exponentially small. The function
$$
f/(\rho \rho ^\dagger )=\cos \left( \frac{\widetilde{I}}{h} \right)\cos
\left( \frac{\widetilde{I}^\dagger}{h}\right)-\frac{1
}{4\rho \rho ^\dagger}e^{-2J/h}
\sin \left( \frac{I}{h}\right)\sin
\left( \frac{I^\dagger}{h}\right) ,
$$
is real-valued on the real axis and has a sequence of local minima
$E_k(\varepsilon ;h)$ such that
$$
E_k(0;h)=E_k^\ell (0;h)+{\cal O}\left(he^{-2J(E_k^\ell (0;h),0)/h} \right),
$$
where $E_k^\ell (0;h)=E_k^r(0;h)$ is defined in (\ref{int.11}),
(\ref{int.12}) and we know that when $|\varepsilon |\le e^{-1/(Ch)}$,
there are two eigenvalues of $P_\varepsilon $ exponentially close to
$E_k(0;h)$ and that we obtain in this way all the eigenvalues in a
fixed neighborhood of $E_0$. It will be convenient to introduce a
``Floquet parameter'' $\kappa \in \mathbb{R}$ and to set 
\begin{align*}
\widetilde{f}(E,\varepsilon &,\kappa ;h) =
\\
&\cos \left( \frac{\widetilde{I}}{h}-\kappa \right)\cos
\left( \frac{\widetilde{I}^\dagger}{h}-\kappa \right)-\frac{e^{-2J/h}
}{4\rho \rho ^\dagger}
\sin \left( \frac{I}{h}-\kappa \right)\sin
\left( \frac{I^\dagger}{h}-\kappa \right).
\end{align*}
We still have a sequence of local minima $E_k(\varepsilon ,\kappa ;h)$
satisfying $E_k(\varepsilon ,\kappa +\pi ;h)=E_{k+1}(\varepsilon ,\kappa
;h)$ and the zeros of $\widetilde{f}(\cdot ,\varepsilon ,\kappa ;h)$ are
now confined, two by two, to exponentially small neighborhoods of
$E_k(\varepsilon ,\kappa ;h)$. We concentrate on one such local minimum
$E_c(\varepsilon ,\kappa ;h)=E_k(\varepsilon ,\kappa ;h)$ and we restrict the attention to a ``window''
\begin{equation}\label{int.20}
E=E_1+hF,\ \ \varepsilon =h\widetilde{\varepsilon },\ \ \kappa
=\widetilde{\kappa }+I(E_1,0)/h,
\end{equation}
where $E_1$ is a real parameter $\in
\mathrm{neigh\,}(E_0,\mathbb{R})$. Assume that $E_c(\varepsilon ,\kappa
;h)$ belongs to the window, so that 
\begin{equation}\label{int.21}
E_c(\varepsilon ,\kappa;h)=E_1+hF_c(\widetilde{\varepsilon } ,\widetilde{\kappa };h),
\end{equation}
where $F_c$ is the corresponding critical point of $\widetilde{f}$ in the variables $F$ with
$\widetilde{\varepsilon }$, $\widetilde{\kappa }$ as the new
parameters. Thanks to the rescaling, this critical point is uniformly
nondegenerate. $\widetilde{f}$ and $F_c$ are even functions of
$\widetilde{\varepsilon }$ and it follows that
\begin{equation}\label{int.22}
F_c(\widetilde{\varepsilon },\widetilde{\kappa
};h)=F_c(0,\widetilde{\kappa };h)+{\cal O}(\widetilde{\varepsilon }^2).
\end{equation}

 We make the assumptions (A1)--(A7) and
  discuss the zeros of $\widetilde{f} $ near the critical value
  $E_c(\varepsilon,\kappa  ;h)=E_1+hF_c(\widetilde{\varepsilon
  },\widetilde{\kappa };h)$. The key point is that, in this regime, we are able to write the quantization condition as a second order polynomial (up to a non-vanishing factor), with a sharp control on the coefficients.

\begin{theorem}\label{int.23}\sl 
The critical value $\widetilde{f}^c(\widetilde{\varepsilon
},\widetilde{\kappa };h)=\widetilde{f}(F_c,\widetilde{\varepsilon
},\widetilde{\kappa };h)$ is of the form
\begin{equation}\label{int.24}
\widetilde{f}^c(\widetilde{\varepsilon },\widetilde{\kappa
};h)=m(\widetilde{\varepsilon },\widetilde{\kappa
};h)(\widetilde{\varepsilon }^2-\widetilde{\varepsilon }
_c(\widetilde{\kappa };h)^2),
\end{equation}
where
\begin{equation}\label{int.25}
\widetilde{\varepsilon }_c=\ell (\widetilde{\kappa };h)e^{-J(E_c(0,\kappa
),1)/h}.
\end{equation}
Here, $\ell$, $m$ are classical symbols of order $0$ as in (\ref{int.5}),
with leading terms satisfying
\begin{equation}\label{int.26}
\ell (\widetilde{\kappa };0)=\frac{1}{2|\partial _\varepsilon I(E_1,0)|},\
\ m(0,\widetilde{\kappa };0)=|\partial _{\varepsilon }I(E_1,0)|.
\end{equation}
Further,
\begin{equation}\label{int.27}
{f}(F,\widetilde{\varepsilon },\widetilde{\kappa
};h)=\widetilde{f}^c(\widetilde{\varepsilon }, \widetilde{\kappa
};h)+q(F,\widetilde{\varepsilon } ,\widetilde{\kappa };h)(F-F_c(\widetilde{\varepsilon },\widetilde{\kappa },1;h))^2,
\end{equation}
where $q$ is a symbol of order $0$ and 
\begin{equation}\label{int.28}
q(F_c,0,\widetilde{\kappa };0)=2(\partial _E\widetilde{I}(E_c(0,\kappa ,0),0))^2.
\end{equation}
\par $\widetilde{f}(\cdot ,\widetilde{\varepsilon },\widetilde{\kappa
};h)$ has two zeros in a small neighborhood of $F_c$ when counted with
their multiplicity, and
\begin{itemize}
\item when $|\widetilde{\varepsilon }|<\widetilde{\varepsilon
  }_c(\widetilde{\kappa };h)$ the zeros are real and simple, given by
\begin{equation}\label{int.29}
q(F,\widetilde{\varepsilon },\widetilde{\kappa
};h)^{\frac{1}{2}}(F-F_c(\widetilde{\varepsilon },\widetilde{\kappa
},1;h))=\pm (-\widetilde{f}^c(\widetilde{\varepsilon },\widetilde{\kappa
};h))^{\frac{1}{2}}
\end{equation}
\item when $|\widetilde{\varepsilon }|=\widetilde{\varepsilon
  }_c(\widetilde{\kappa };h)$ we have a double zero, 
\begin{equation}\label{int.30}
F=F_c.
\end{equation}
\item when $|\widetilde{\varepsilon }|>\widetilde{\varepsilon
  }_c(\widetilde{\kappa };h)$ the zeros are simple, non-real and
  complex conjugate to each other, given by
\begin{equation}\label{int.31}
q(F,\widetilde{\varepsilon },\widetilde{\kappa
};h)^{\frac{1}{2}}(F-F_c(\widetilde{\varepsilon },\widetilde{\kappa
},1;h))=\pm i (\widetilde{f}^c(\widetilde{\varepsilon },\widetilde{\kappa
};h))^{\frac{1}{2}}.
\end{equation}
\end{itemize}
When $\kappa =\widetilde{\kappa }+I(E_1,0)/h$ belongs to $\pi
\mathbb{Z}$, these values give the eigenvalues of $P_\varepsilon $ near
$E_c(\varepsilon ,\kappa ;h)$ via (\ref{int.20}).
\end{theorem}

Eventually, we would like to mention the paper \cite{GeGr88} by C. G\'erard and A. Grigis,  where the authors study the eigenvalues of self-adjoint Schr\"odinger operators with a double well potential. They also obtain a quantization condition, using what they call the "exact WKB method". Our method here is slightly different and more explicit about the connection formulas at the turning points.

\section{The complex WKB method}
\label{wkb}

We recall here briefly elements of the complex WKB method in a general setting. Consider a Schr\"{o}dinger equation,
\begin{equation}
-h^2u''(z,h) + V(z) u(z,h) =0,
\label{wkb.1}
\end{equation}
in a bounded, simply connected  open set $U\subset \C$ where the potential $V$ is holomorphic. 
We look for a solution of the type
\begin{equation}
u(z,h)=a(z,h)e^{i\varphi (z)/h}, 
\label{wkb.2}
\end{equation}
where $a(z,h)$ has a formal asymptotic expansion in a sense to be defined later on,
\begin{equation}
a(z,h)\sim \sum\limits_{j=0}^{+\infty }a_{j}(z)h^{j},
\label{wkb.3}
\end{equation}
and the $a_{j}$'s are holomorphic functions in $U$. The function $\varphi$ is called the phase of the solution $u$, and $a(z,h)$ is called its symbol.

A function $u(z,h)$ of the form (\ref{wkb.2}) is a solution to (\ref{wkb.1}) if and only if
\begin{equation}
e^{-i\varphi (z)/h}(-h^{2}\partial _{z}^{2}+V(z))(e^{i\varphi (z)/h}a(z,h))=0,
\label{wkb.4}
\end{equation}
or
\begin{equation}
(-(h\partial _{z})^{2}-2i\varphi'(z)h\partial _{z}-ih{\varphi''}(z)+\varphi'(z)^{2}+V(z))a(z,h) =0.
\end{equation}
If $\varphi$ is a solution of the  eikonal equation
\begin{equation}
\varphi'(z)^{2}+V(z)=0, 
\label{wkb.5}
\end{equation}
then (\ref{wkb.4}) is equivalent to
\begin{equation}
\left (
\varphi' (z)\partial _{z}+\frac{\varphi'' (z)}{2}-\frac{ih}{2}\partial _{z}^{2}
\right )
a(z,h) =0.
\label{wkb.6}
\end{equation}
Replacing $a(z,h)$ by its formal asymptotic expansion (\ref{wkb.3}), and canceling successively the powers of $h$, we obtain a sequence of transport equations
\begin{equation}
\label{wkb.7}
\left\{
\begin{array}{l}
\ds
\left (\varphi' (z)\partial _{z}+\frac{1}{2}\varphi''(z)\right )a_{0}=0,
\\[10pt]
\ds 
\left (\varphi'(z) \partial _{z}+\frac{1}{2}\varphi''(z) \right )a_{j} 
=
\frac{i}{2}a''_{j-1}, \mbox{ for } j\geq 1.
\end{array}
\right .
\end{equation}

\begin{definition}\sl
A formal WKB solution $u_{wkb}$ of the equation (\ref{wkb.1}) in $U$ is a pair $(\varphi,(a_{j}))$ of an analytic function $\varphi$ in $U$ verifying the eikonal equation (\ref{wkb.5}), and of a sequence $(a_{j})$ of analytic functions in $U$ which satisfies the transport equations (\ref{wkb.7}). We denote it 
\begin{equation}
\label{wkb.8}
u_{wkb}(z,h)=e^{i\varphi(z)/h}\sum_{j\geq 0}a_{j}(z)h^j.
\end{equation}
\end{definition}

We suppose from now on that $V(z)\neq 0$ for all $z\in U$. Then we fix a determination of $z\mapsto (-V(z))^{\frac{1}{2}}$ in $U$, and we solve the eikonal and transport equations in $U$.

\begin{proposition}\sl
\label{wkb.9}
The solutions of the eikonal equation (\ref{wkb.5}) are analytic functions in $U$, and they 
 can be written
\begin{equation}
\varphi (z)=\pm \int_{z_{0}}^{z}(-V(w))^{\frac{1}{2}}dw  +C
\label{wkb.10}
\end{equation}
for some $z_{0}\in U$, and some $C\in \C$.
\end{proposition}

Now we  fix a  such a solution $\varphi$ in $U$. It is then easy to prove by induction that, given an initial data, the system of transport equations  has a unique solution. Therefore we have the
\begin{proposition}\sl
Let $(a_{j}^{0})_{j=0}^\infty $ be any sequence of complex numbers.  Then  the Schr\"odinger equation (\ref{wkb.1}) has a unique formal WKB solution in $U$, such that
\begin{equation*}
\forall j\in \N,\; a_{j}(z_{0})=a_{j}^{0}.
\end{equation*}
Moreover, the function $a_{0}$ is given in $U$, for some suitable constant $C\in \C$, by
$$
a_{0}(z)=C(\varphi'(z))^{-1/2}.
$$
\end{proposition}

We want now to associate true solutions of the Schr\"odinger equation (\ref{wkb.1}) to the formal ones we have constructed above. It is convenient to introduce the  notion of Stokes line for the potential $V$.

\begin{definition}\sl
\label{wkb.11}
Let $U$ be a simply connected open set in $\C$ where $V$ is holomorphic. A Stokes line
 is a $\CC^{1}$ curve $\sigma :I\rightarrow U$ such that
\begin{equation*}
\im\int_{s}^{t}(-V(\sigma (\tau )))^{\frac{1}{2}}\sigma' (\tau )d\tau=0,
\end{equation*}
for all $s,t\in I$. Here, $I$ is any interval starting at $0$ and ending at $1$.
\end{definition}

Notice that 
$$
\im\int_{s}^{t}(-V(\sigma (\tau )))^{\frac{1}{2}}\sigma' (\tau )d\tau
=\im (\varphi(\sigma(t)))-\im (\varphi(\sigma(s))),
$$
where $\varphi$ is a solution of the eikonal equation. Thus, a Stokes line is nothing else than a level curve in $U$ of the imaginary part of the phase $\varphi$.

The following proposition is well known (see for example \cite{Sj14+} for a proof), and can be considered as the fundamental rule of the complex WKB method: always move in a direction where the modulus of the phase factor increases, thus in particular transversely to the Stokes lines.

\begin{proposition}\sl\label{wkb.12} 
Let $U$ be a simply connected bounded open subset of $\C$, such that $V(z)\neq 0$ for all $z\in U$. Let $\varphi$ be a solution of the eikonal equation in $U$. Let also $\gamma: ]0,1[ \rightarrow U$ be a $\CC^{1}$ curve in $U$ such that 
\begin{equation}
\forall t\in ]0,1[,\; \frac{d}{dt}(-\im\varphi(\gamma (t)))>0,
\end{equation}
Then there exists a neighborhood $\Omega\subset U$  of $\gamma$ such that,
for any formal WKB solution
$u_{wkb}(z,h)=e^{i\varphi(z)/h}\sum_{j\geq 0}a_{j}(z)h^j$, there exists a solution $u$ of the Schr\"{o}dinger equation (\ref{wkb.1}) in $\Omega$ such that
\begin{equation}
\label{wkb.13}
u(z,h)=e^{i\varphi(z)/h}a(z,h),
\end{equation}
where $a$ is holomorphic with respect to $z\in \Omega$, and
\begin{equation}
\label{wkb.14}
a(z,h)\sim \sum_{k\geq 0}a_{j}(z)h^j \mbox{ in} \Hol(\Omega).
\end{equation}

\end{proposition}

\section{WKB analysis near a simple turning point}
\label{st}

In this section we follow closely the
presentation in \cite{Sj14+}. Let $\Omega \subset \mathbb{C}$
be open and simply connected, and $V\in \mathrm{Hol\,}(\Omega )$. We suppose that $V$ has a unique zero $z_{0}$ in $\Omega$, and that it is a simple one: 
\begin{equation}
V(z_{0})=0,\ V'(z_{0})\neq 0.
\label{st.1}
\end{equation}
We are
interested in solutions $u$ in $\Omega$ of the general Schr\"odinger equation (\ref{wkb.1}) of the form 
\begin{equation}
\label{st.2}
u(z,h)= a(z,h)e^{\varphi (z)/h}.
\end{equation}
Notice that, contrary to (\ref{wkb.2}), here we have chosen not to put the factor $i$ in the exponent to simplify the notations. For the same reason, we will also assume that $z_{0}=0$.

As in Section \ref{wkb}, we obtain first the eikonal equation
\begin{equation}
\label{st.3}
\varphi'(z)=V(z)^{\frac{1}{2}},
\end{equation}
in $\Omega$. By assumption there exists a function $F$, holomorphic  in $\Omega$, such that
$$
V(z)=zF(z)
$$
and $F(z)\neq 0$ for $z\in \Omega$ (we may decrease $\Omega$ whenever
necessary). It is therefore clear that $\varphi (z)$ is multi-valued in general, and to
better understand the structure of this singularity we pass to the
double covering $\widetilde \Omega$ of $\Omega \setminus \{0\}$, setting $z=w^{2}$. Then 
\begin{equation*}
\frac{\partial }{\partial z}=\frac{1}{2w}\frac{\partial }{\partial w},
\end{equation*}
and if we set
$$
\left\{ 
\begin{array}{l}
\widetilde{V}(w)=V(z)=w^{2}F(w^{2}),\\[8pt]
\widetilde{\varphi }(w)=\varphi (z),
\end{array}
\right .
$$
the eikonal equation becomes 
\begin{equation*}
\partial _{w}\widetilde{\varphi }(w)=2w^{2}F(w^{2})^{\frac{1}{2}}.
\end{equation*}%
Notice that the right hand side is an even holomorphic function. If we also require
that $\varphi (0)=\widetilde{\varphi }(0)=0$, we see that $\widetilde{\varphi }(w)$ 
is an odd holomorphic function of the form 
\begin{equation*}
\widetilde{\varphi }(w)=\frac{2}{3}\widetilde{F}(w^{2})w^{3},\mbox{ where }
\widetilde{F}(0)=F(0)^{\frac{1}{2}}=V^{\prime }(0)^{\frac{1}{2}}.
\end{equation*}
In the original coordinates, we get the double-valued solution, 
\begin{equation}
\varphi (z)=\frac{2}{3}\widetilde{F}(z)z^{\frac{3}{2}}. 
\label{st.4}
\end{equation}
Now we study the Stokes and anti-Stokes lines having $0$ as a limit point. Since, with respect to Section \ref{wkb}, we
have removed the factor $i$ in the exponent in (\ref{st.2}), Stokes lines are now level curves of the real part of $\varphi$, and level curves of $\im \varphi$ are called anti-Stokes lines. On such curves we have $\re\varphi =0$ or 
$\im\varphi =0$, which is equivalent to  $\im\varphi
^{2}=0$, and to 
$
\im z^{3}\widetilde{F}(z)^{2}=0$.
In other words, these curves are given by
\begin{equation}\label{st.5}
\{z\in \Omega,\; \exists t\in\R, \; z^{3}{\widetilde F}(z)^{2}=t^{3}\}.
\end{equation}
Taking the cubic root, we see that Stokes and anti-Stokes lines reaching 0 in the limit, are contained in three curves $\gamma _{k}$ 
given by 
\begin{equation}\label{st.6}
\gamma_{k}=\{z\in \Omega,\; \exists t\in\R,\; z\widetilde{F}(z)^{\frac{2}{3}}=e^{2\pi ik/3}t\},\; k\in \{0,1,2\}\simeq 
\Z/3\Z.
\end{equation}
In the case where $V'(0)>0$, the situation is as shown in Figure \ref{trfig1}. Each curve $\gamma _{k}\setminus \{0\}$ is divided into a Stokes line $\gamma _{k}^{-}$ (plain lines) and an anti-Stokes line $\gamma _{k}^{+}$ (dashed lines). The three
Stokes lines delimit three closed Stokes sectors $\Sigma _{k}$, $k\in \Z/3\Z$, where $\Sigma _{k}$ is the sector that contains $\gamma_{k}^+$. In 
Figure \ref{trfig1}, we have also drawn a Stokes line inside each sector.

\def\svgwidth{9cm}
\begin{figure}[!h]
\begin{center}
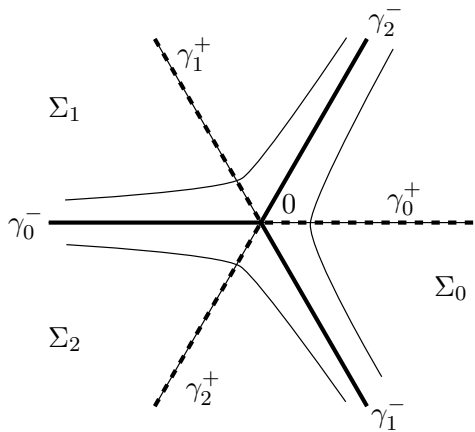
\caption{Stokes lines close to a simple turning point 
\label{trfig1}
}
\end{center}
\end{figure}

For $k\in \Z/3\Z$, we denote by $\varphi _{k}$ the branch of $\varphi $ in $\Omega \setminus \gamma
_{k}^{-}$ such that  $\varphi _{k}(0)=0$, and $\re\varphi _{k}<0$ in $\mathring\Sigma
_{k}$. Notice that $\varphi _{k+1}$ and $\varphi _{k}$
are both well defined in $\Sigma _{k}\cup \Sigma _{k+1}$ and satisfy
\begin{equation}
\label{st.7}
\varphi _{k+1}=-\varphi _{k} \mbox{ in } \Sigma _{k}\cup \Sigma _{k+1}.
\end{equation}
According to  Proposition \ref{wkb.12}, there are
solutions $u=u_{k}$, $k\in \Z/3\Z$, of the Schr\"odinger equation (\ref{wkb.1}) in $\Omega $ such that, in 
$\mathring\Sigma_{k}$,
\begin{equation}
\left\{
\begin{array}{l}
\ds u_{k}(z,h)=a_{k}(z,h)e^{\varphi _{k}(z)/h}, \\[8pt] 
\ds a_{k}(z,h)\sim \sum_{j\geq 0}a_{k,j}(z)h^j \mbox{ in } \Hol(\mathring\Sigma_{k}).
\end{array}
\right .
\label{st.8}
\end{equation}
This asymptotic description extends to $\Omega_{k}$, the complement of an arbitrarily
small neighborhood of $\gamma _{k}^{-}\cup \{0\}$, that can be reached from $\Sigma _{k}$ by crossing the
Stokes lines transversally. We also recall that $a_{k,0}$ is unique up to a
constant factor and that we can choose
\begin{equation}
a_{k,0}(z)=(\varphi _{k}'(z))^{-\frac{1}{2}},
\label{st.9}
\end{equation}
for any branch of the square root.

Recall that if $u,v$ are solutions to the Schr\"{o}dinger
equation, then the Wronskian 
\begin{equation*}
W_{h}(u,v)=(h\partial_{z} u)v-u(h\partial_{z} v).
\end{equation*}
is constant, and vanishes precisely when $u$, $v$ are collinear. Let $j,k\in \Z/3\Z$. Applying the asymptotics of $u_{j}$ and $u_{k}$ at some point
in the interior of $\Sigma _{j}\cup \Sigma _{k}$, we see that, recalling (\ref{st.9}),  $W_{h}(u_{j},u_{k})$ has an asymptotic expansion in powers of $h$, whose first term is given by 
\begin{equation}\label{st.10}
  W_{h}(u_{j},u_{k})
  =
  2a_{j,0}a_{k,0}\varphi'_{j}+\mathcal{O}(h)
  =
  \ds\frac{2(\varphi'_{j})^{\frac{1}{2}}}{(\varphi_{k}')^{\frac{1}{2}}}+\mathcal{O}(h).
\end{equation}
We  fix a branch of
$(\varphi_{k}')^{\frac{1}{2}}$ in $\Sigma _{k}$ for each $k\in
\Z/3\Z$. For two different Stokes sectors, $\Sigma_{j}\neq
\Sigma_{k}$, we have in the interior of $\Sigma _{j}\cup \Sigma _{k}$,
that
\begin{equation}
(\varphi _{j}^{\prime })^{\frac{1}{2}}=i^{\nu _{j,k}}(\varphi _{k}^{\prime })^{\frac{1}{2}},
\label{st.11}
\end{equation}%
for some $\nu _{j,k}\in {\Z}/4{\Z}$ which are odd, and such that $\nu _{j,k}=-\nu
_{k,j}$.
Thus, starting from $\Sigma _{0}$, we can make a tour around $0$ in the positive direction and
we get that 
\begin{equation}
\begin{split}
(\varphi _{1}')^{\frac{1}{2}}& =i^{\nu _{1,0}}(\varphi _{0}^{\prime })^{\frac{1}{2}} \mbox{ in } \Sigma_{1},
\\
(\varphi _{2}')^{\frac{1}{2}}& =i^{\nu _{2,1}}(\varphi _{1}^{\prime })^{\frac{1}{2}} \mbox{ in } \Sigma_{2},
\\
(\varphi _{0}')^{\frac{1}{2}}& =i^{\nu _{0,2}}(\varphi _{2}^{\prime })^{\frac{1}{2}} \mbox{ in } \Sigma_{0}.
\end{split}
\label{st.12}
\end{equation}
This means that if we follow a continuous branch of $(\varphi _{0}^{\prime
})^{\frac{1}{2}}$ around $0$ in the positive direction, then after a turn, we obtain
the new branch
\begin{equation}
i^{-(\nu _{0,2}+\nu _{2,1}+\nu _{1,0})}(\varphi _{0}^{\prime })^{\frac{1}{2}}.
\label{st.13}
\end{equation}
But $(\varphi_{0}')^{\frac{1}{2}}=V^{1/4}$ for a suitable branch of the fourth
root, and if one follows this function around $0$ once in the positive direction, we obtain $iV^{1/4}$.
This gives the co-cycle condition
\begin{equation}
\nu _{0,2}+\nu _{2,1}+\nu _{1,0}\equiv -1\ \mathrm{mod\,}4.
\label{st.14}
\end{equation}(\ref{st.10}) and (\ref{st.11}) imply that
\begin{eqnarray}
\nonumber
W_{h}(u_{j},u_{k})
=2i^{\nu
_{j,k}}+\mathcal{O}(h).
\label{tr314}
\end{eqnarray}

Now we describe the linear space of solutions of the Schr\"odinger equation (\ref{wkb.1}) in $\Omega$. It is
of course of dimension $2$, and any two of $u_{-1}$, $%
u_{0}$, $u_{1}$ are linearly independent, so we have a relation
\begin{equation}
\alpha _{-1}u_{-1}+\alpha _{0}u_{0}+\alpha _{1}u_{1}=0,
\label{st.15}
\end{equation}
where the vector $(\alpha _{-1},\alpha _{0},\alpha _{1})^{T}\in 
\C^{3}\setminus\{0\}$ is well defined up to a scalar factor.
Applying $W(u_{j},\cdot )$ to this relation, we get
\begin{equation}
(W(u_{j},u_{k}))_{j,k}
\begin{pmatrix}
\alpha _{-1} \\ 
\alpha _{0} \\ 
\alpha _{1}
\end{pmatrix}
=0,  
\label{st.16}
\end{equation}
which is a system of the form
\begin{equation}
\begin{pmatrix}
0 & a & b \\ 
-a & 0 & c \\ 
-b & -c & 0
\end{pmatrix}
\begin{pmatrix}
\alpha _{-1} \\ 
\alpha _{0} \\ 
\alpha _{1}
\end{pmatrix}
=0.
\label{st.17}
\end{equation}
The triplet $(\alpha _{-1},\alpha _{0},\alpha _{1})=(c,-b,a)$ is a solution, so
up to a common factor, we have
\begin{equation}
\alpha _{j}=\pm i+\mathcal{O}(h).
\end{equation}
More precisely, the values of $a,b$ and $c$ are given by the equation (\ref{tr314}), and we get, after inserting a factor $1/2$,
\begin{equation}
\begin{pmatrix}
\alpha _{-1} \\ 
\alpha _{0} \\ 
\alpha _{1}
\end{pmatrix}
=
\begin{pmatrix}
i^{\nu _{0,1}} \\ 
-i^{\nu _{-1,1}} \\ 
i^{\nu _{-1,0}}
\end{pmatrix}
+\CO(h)=
\begin{pmatrix}
i^{\nu _{0,1}} \\ 
i^{\nu _{1,-1}} \\ 
i^{\nu _{-1,0}}
\end{pmatrix}
+\CO(h).
\label{st.18}
\end{equation}

\begin{remark}\sl
\label{st.19} 
Sometimes it is more natural to change the notation, writing $i\varphi _{j}$ in (\ref{st.8}) instead of $\varphi _{j}$ so that $u_{j}(z;h)=a_{j}(z;h)e^{i\varphi _{j}(z)/h}$ with $\im\varphi
_{j}\geq 0$ in $\Sigma _{j}$. Then (\ref{st.9}) becomes $a_{j,0}(z)=(i
\varphi _{j}^{\prime })^{-1/2}=V(z)^{-1/4}$ and in (\ref{st.11}), (\ref
{st.12}), $\varphi_{j}'$ must be replaced  by $i\varphi_{j}'$.
\end{remark}

\section{WKB solutions near the wells}
\label{sw}

From now on, we consider the equation 
$P_{h,\varepsilon}u=Eu$, that is
\begin{equation}
\label{sw.1}
-h^2u'' +(V_{\varepsilon}(x)-E)u =0,
\end{equation}
where $V_{\varepsilon}=V_{0}+i\varepsilon W$ satisfies  (A1) to (A4)
and (A6). For the moment we do not assume the $\CP\CT$-symmetry property (A5).

Let us now define some formal WKB solutions to the Schr\"odinger equation (\ref{sw.1}) near the wells. For $(E,\varepsilon)\in D(E_{0},\varepsilon_{0})\times D(0,\varepsilon_{0})$, the equation $V_{\varepsilon}(x)=E$ has exactly four solutions in $U$, the domain of holomorphy of $V_{\varepsilon}$, that are called turning points at energy $E$. We have denoted them
$\alpha_{\ell}(E,\varepsilon)$, $\beta_{\ell}(E,\varepsilon)$,
$\beta_{r}(E,\varepsilon)$ and $\alpha_{r}(E,\varepsilon)$, with,
\begin{equation}\label{sw.2}
\alpha_{\bullet}(E_{0},0)=\alpha_{\bullet}, \quad \beta_{\bullet}(E_{0},0)=\beta _{\bullet},\ \bullet=\ell,r.
\end{equation}
We have drawn in Figure
\ref{sw.fig.1} a typical configuration of the Stokes lines starting at
each of the turning points, when $E\neq E_{0}$ and $\varepsilon\neq
0$. 

\begin{figure}[!h]
\begin{center}
\def\svgwidth{12cm}
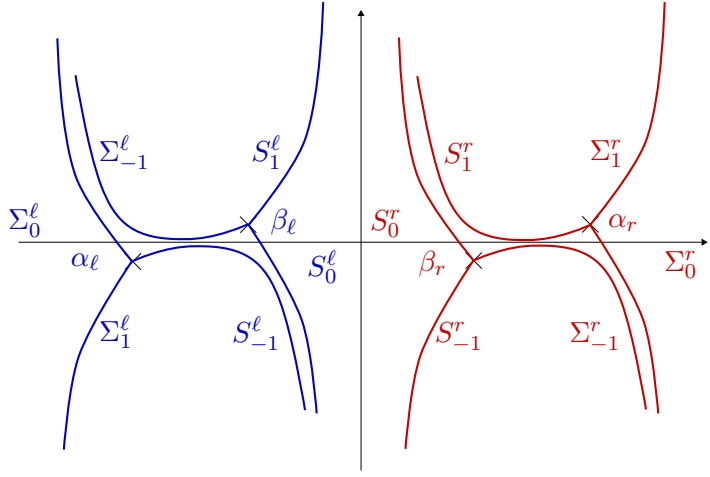
\end{center}
\caption{Stokes lines and Stokes sectors \label{sw.fig.1}}
\end{figure}

We shall work in the cut complex plane along $[\alpha_{\ell},\beta_{\ell}]\cup [\beta_{r},\alpha_{r}]$,  or more precisely in the cut version $\tilde U$ of $U$, so that we have two determinations of  $x\mapsto (V_{\varepsilon}(x)-E)^{\frac{1}{2}}$ in $\tilde U$. We denote
$$
x\mapsto (V_{\varepsilon}(x)-E)_{\ell}^{\frac{1}{2}},
(\mbox {resp.}\  x\mapsto (V_{\varepsilon}(x)-E)_{m}^{\frac{1}{2}}, \  x\mapsto (V_{\varepsilon}(x)-E)_{r}^{\frac{1}{2}}),
$$
the determination  which is real and positive for $\varepsilon=0$, $E=E_{0}$ and $x\in ]-\infty, \alpha_{\ell}[$ (resp. $x\in ]\beta_{\ell}, \beta_{r}[$, $x\in ]\alpha_{r},+\infty[$). Notice that
$$
\forall x\in\tilde U,
\ (V_{\varepsilon}(x)-E)_{\ell}^{\frac{1}{2}} =
 (V_{\varepsilon}(x)-E)_{r}^{\frac{1}{2}}=
 -(V_{\varepsilon}(x)-E)_{m}^{\frac{1}{2}}.
$$

First, we concentrate on the situation near the left well. We denote
$\Sigma^\ell_{0}$, $\Sigma^\ell_{1}$ and $\Sigma^\ell_{-1}$ the three
Stokes sectors near $\alpha_{\ell}$, and $S^\ell_{0}$, $S^\ell_{1}$
and $S^\ell_{-1}$ those near $\beta_{\ell}$. For each sector
$\Sigma^\ell_k$ (resp. $S^\ell_{k}$), $k\in \Z/3\Z$, we choose a
solution $u^\ell_{k}$ (resp. $v^\ell_{k}$) of (\ref{sw.1}) such that
\begin{equation}\label{sw.3}\begin{split}
u^\ell_{k}(z,E,\varepsilon,h)&=a^\ell_{k}(z,E,\varepsilon,h)e^{i\varphi^\ell_{k}(z,E,\varepsilon)/h}
\hbox{ in }\Sigma ^\ell_k\\
v^\ell_{k}(z,E,\varepsilon,h)&=b^\ell_{k}(z,E,\varepsilon,h)e^{i\psi^\ell_{k}(z,E,\varepsilon)/h}
\hbox{ in }S ^\ell_k .
\end{split}
\end{equation}
Here $\varphi^\ell_{k}$ (resp. $\psi^\ell_{k}$)  is a solution of the eikonal equation
\begin{equation}
\label{sw.4}
(i\varphi'(x))^2=V_{\varepsilon}(x)-E,
\end{equation}
vanishing at $z=\alpha _\ell (E,\varepsilon )$ (resp. at $\beta _\ell
(E,\varepsilon )$) for $(E,\varepsilon)\in
D(E_{0},\varepsilon)\times D(0,\varepsilon_{0})$, such that
\begin{equation}\label{sw.5}\begin{split}
&\forall z\in \Sigma^\ell_{k},\; \re (i\varphi^\ell_{k}(z,E,\varepsilon))<0,\\
&\forall z\in S^\ell_{k},\; \re (i\psi^\ell_{k}(z,E,\varepsilon))<0.
\end{split}
\end{equation}
The amplitudes $a^\ell_k$ and $b^\ell_k$ in (\ref{sw.3}) have
asymptotic expansions in $\mathrm{Hol\,}(\mathring{\Sigma } _k^\ell)$
and $\mathrm{Hol\,}(\mathring{S} _k^\ell)$ respectively, in the sense
of (\ref{int.5}) The phase functions $\varphi^\ell_{k}$ and
$\varphi^\ell_{k+1}$ are well defined in $\Sigma^\ell_{k} \cup
\Sigma^\ell_{k+1}$, and $\psi^\ell_{k}$, $\psi^\ell_{k+1}$ are well
defined in $S^\ell_{k} \cup S^\ell_{k+1}$, where they satisfy
\begin{equation}\label{sw.6}
\varphi^\ell_{k}=-\varphi^\ell_{k+1}, \quad \psi^\ell_{k}=-\psi^\ell_{k+1}.
\end{equation}
We choose the functions $u^\ell_{k}$ (resp. $v^\ell_{k}$) so that they
are holomorphic also with respect to $(E,\varepsilon )$ in
$D(E_{0},\varepsilon_{0})\times D(0,\varepsilon_{0})$. The asymptotic
expansions of the amplitudes $a^\ell_k$ and $b^\ell_k$ extend to 
$\Omega ^\ell_{k} \times D(E_{0},\varepsilon_{0})\times
D(0,\varepsilon_{0})$ and 
$\CO^\ell_{k} \times D(E_{0},\varepsilon_{0})\times
D(0,\varepsilon_{0})$ respectively, where $\Omega^\ell_{k}$ (resp. $\CO^\ell_{k}$)
is the complement of an arbitrarily small neighborhood of
$\gamma_{\alpha_{\ell},k}^{-}$ (resp. $\gamma_{\beta_{\ell},k}^{-}$) in
$U$ (see Figure \ref{trfig1}).

We fix now a choice for the principal symbols $a^\ell_{k,0} $ and $b_{k,0}^\ell$ of these six solutions $u^\ell_{k}$ and $v^\ell_{k}$, $k\in \Z/3\Z$.

By (\ref{sw.4}) and (\ref{sw.5}) we  have
\begin{equation}\label{sw.7}
\left\{
\begin{array}{l}
\ds i\varphi^\ell_{0}(x,E,\varepsilon) = \int_{\alpha_{\ell}}^x (V_{\varepsilon}(t)-E)_{\ell}^{\frac{1}{2}} dt,\\[8pt]
\ds i\varphi^\ell_{-1}(x,E,\varepsilon) = -\int_{\alpha_{\ell}}^x (V_{\varepsilon}(t)-E)_{\ell}^{\frac{1}{2}} dt,\\[8pt]
\ds i\varphi^\ell_{1}(x,E,\varepsilon) = -\int_{\alpha_{\ell}}^x (V_{\varepsilon}(t)-E)_{\ell}^{\frac{1}{2}} dt
.
\end{array}
\right .
\end{equation}
 Then we choose the principal symbol $a^\ell_{0,0}$ of $a^\ell_{0}$ to be
\begin{equation}\label{sw.8}
a^\ell_{0,0}(x,E,\varepsilon)=[(i\varphi^\ell_{0})']^{-\frac{1}{2}}=(V_{\varepsilon}(x)-E)_{\ell}^{-\frac{1}{4}}.
\end{equation}
In order to fix the principal symbol of $u^\ell_{1}$ and $u^\ell_{-1}$, we have to choose $\nu_{0,1},\nu_{1,-1}$ and $\nu_{-1,0}$ in $\Z/4\Z$, odd,  such that (\ref{st.14}) holds. We take
\begin{equation}
\nu_{0,1}=1,\ \nu_{1,-1}=-1, \mbox{ and } \nu_{-1,0}=1.
\label{sw.9}
\end{equation}
Then by (\ref{st.12}) we have,
\begin{equation}\label{sw.10}
a^\ell_{-1,0}(x,E,\varepsilon)=\frac{1}{i^{\nu_{-1,0}}}a^\ell_{0,0}(x,E,\varepsilon)= -i (V_{\varepsilon}(x)-E)_{\ell}^{-\frac{1}{4}},
\end{equation}
and,
\begin{equation}\label{sw.11}
a^\ell_{1,0}(x,E,\varepsilon)=\frac{1}{i^{\nu_{1,0}}}a^\ell_{0,0}(x,E,\varepsilon)=i (V_{\varepsilon}(x)-E)_{\ell}^{-\frac{1}{4}}.
\end{equation}
With these choices, we have 
\begin{equation}\label{sw.11.5}
u_{0}^\ell=\tau_{+}(h)u_{-1}^\ell + \tau_{-}(h)u_{1}^\ell,
\end{equation}
for some symbols $\tau_{\pm}(h)$ with
$\tau_{\pm}(h)=1+\CO(h)$. Without changing the leading asymptotics, we
can replace $u^\ell_{\mp 1}$ by $\tau _\pm^\ell u_{\mp 1}^\ell$ and we get
\begin{equation}\label{sw.12}
u_{0}^\ell=u_{-1}^\ell + u_{1}^\ell.
\end{equation}

Let us now consider the solutions $v^\ell_{k}$, $k\in \Z/3\Z$, near
$\beta_{\ell}$. By (\ref{sw.4}) and (\ref{sw.5}) we have
\begin{equation}\label{sw.13}
\left\{
\begin{array}{l}
\ds i\psi^\ell_{0}(x,E,\varepsilon) =-\int_{\beta_{\ell}}^x (V_{\varepsilon}(t)-E)_{m}^{\frac{1}{2}}dt,\\[8pt]
\ds i\psi^\ell_{-1}(x,E,\varepsilon)=\int_{\beta_{\ell}}^x (V_{\varepsilon}(t)-E)_{m}^{\frac{1}{2}} dt,\\[8pt]
\ds i\psi^\ell_{1}(x,E,\varepsilon)=\int_{\beta_{\ell}}^x (V_{\varepsilon}(t)-E)_{m}^{\frac{1}{2}} dt
.
\end{array}
\right .
\end{equation}
We here consider $-z$ as the basic variable for the Schr\"odinger
equation and correspondingly, we choose the principal symbol $b^\ell_{0,0}$ of $b^\ell_{0}$ to  be
\begin{equation}\label{sw.14}
b^\ell_{0,0}(x,E,\varepsilon)=[(-i\psi^\ell_{0})']^{-\frac{1}{2}}=(V_{\varepsilon}(x)-E)_{m}^{-\frac{1}{4}},
\end{equation}
and we fix the principal symbol of $v^\ell_{1}$ and $v^\ell_{-1}$,
choosing $\nu_{0,1},\nu_{1,-1}$ and $\nu_{-1,0}$ as in (\ref{sw.9}):
\begin{equation}
\nu_{0,1}=1,\ \nu_{1,-1}=-1, \mbox{ and } \nu_{-1,0}=1.
\label{sw.15}
\end{equation}
We get, 
\begin{equation}\label{sw.16}
b^\ell_{-1,0}(x,E,\varepsilon)=\frac{1}{i^{\nu_{-1,0}}}b^\ell_{0,0}(x,E,\varepsilon)= -i (V_{\varepsilon}(x)-E)_{m}^{-\frac{1}{4}},
\end{equation}
and
\begin{equation}\label{sw.17}
b^\ell_{1,0}(x,E,\varepsilon)=\frac{1}{i^{\nu_{1,0}}}b^\ell_{0,0}(x,E,\varepsilon)=i (V_{\varepsilon}(x)-E)_{m}^{-\frac{1}{4}}.
\end{equation}

We further fix a choice of $v^\ell_{\pm 1}$. The principle of the WKB
method ensures that we can choose $v^{\ell}_{1}$ to be proportional to
$u^{\ell}_{-1}$, and $v^{\ell}_{-1}$ to be proportional to
$u^{\ell}_{1}$.  Notice first that
$$
i\psi^\ell_{1}(x,E,\varepsilon)-i\varphi^\ell_{-1}(x,E,\varepsilon)= 
\int_{\alpha_\ell}^{\beta_{\ell}}
(V_{\varepsilon}(t)-E)_{\ell}^{\frac{1}{2}}dt =-iI_{\ell}(E,h),
$$
where we have set
\begin{equation}\label{sw.18}
I_{\ell}(E,\varepsilon)=\int_{\alpha_\ell}^{\beta_{\ell}}(E-V_{\varepsilon}(t))_{w}^{\frac{1}{2}}\, dt
\end{equation}
where $(E-V_{\varepsilon}(t))_{w}^{\frac{1}{2}}$ is real and positive for $\varepsilon=0$, $E\in \R$  close to $E_{0}$ and $t\in ]\alpha_{\ell},\beta_{\ell}[$. The same way, we see that
$$
i\psi^\ell_{-1}(x,E,\varepsilon)-i\varphi^\ell_{1}(x,E,\varepsilon)= iI_{\ell}(E,h).
$$
Concerning the principal symbols, we have first
\begin{equation}\label{sw.19}
\begin{split}
&(V_{\varepsilon}(x)-E)_{\ell+}^{\frac{1}{4}}= -i(V_{\varepsilon}(x)-E)_{m}^{\frac{1}{4}},
\\
&(V_{\varepsilon}(x)-E)_{\ell-}^{\frac{1}{4}}= i(V_{\varepsilon}(x)-E)_{m}^{\frac{1}{4}}.
\end{split}
\end{equation}
Here we have denoted $(V_{\varepsilon}(x)-E)_{l+}^{-\frac{1}{4}}$  (resp. $(V_{\varepsilon}(x)-E)_{l-}^{-\frac{1}{4}}$) the determination of $(V_{\varepsilon}(x)-E)^{-\frac{1}{4}}$ obtained on $]\beta_\ell,\beta_r[$ by extending $(V_{\varepsilon}(x)-E)_{\ell}^{-\frac{1}{4}}$ on $\tilde U$ along a path in the upper (resp. lower) half plane.
Thus we see that
\begin{equation}\label{sw.20}
b^\ell_{1,0}=ia^\ell_{-1,0} \ \mbox{ and }\  b^\ell_{-1,0}=-ia^\ell_{1,0},
\end{equation} 
and we can assume that
\begin{equation*}
\left\{
\begin{array}{l}
\ds u_{-1}^\ell=-i e^{iI_{\ell}(E,\varepsilon)/h}\sigma^\ell_{+}v_{1}^\ell,
\\[8pt]
\ds u_{1}^\ell= i e^{-iI_{\ell}(E,\varepsilon)/h}\sigma^\ell_{-}v_{-1}^\ell,
\end{array}
\right .
\end{equation*}
for some symbols $\sigma^\ell_{\pm}$ such that
$\sigma^\ell_{\pm}=1+\CO(h)$. After replacing $v_{\pm 1}^\ell $ by
$\sigma _{\pm }^\ell v_{\pm 1}^\ell$ (which does not change the
leading asymptotics) we may assume that $\sigma ^\ell_{\pm}=1$: 
\begin{equation}\label{sw.21}
\left\{
\begin{array}{l}
\ds u_{-1}^\ell=-i e^{iI_{\ell}(E,\varepsilon)/h}v_{1}^\ell,
\\[8pt]
\ds u_{1}^\ell= i e^{-iI_{\ell}(E,\varepsilon)/h}v_{-1}^\ell,
\end{array}
\right .
\end{equation}

The same discussion applies to the solutions associated to the well to the right. 
The main rule is simply that the right well becomes the left
well of the operator $\widetilde{P}_\varepsilon= -h^2\partial
_{\widetilde{x}}^2+V_\varepsilon (-\widetilde{x})$ under the change of
variables $x=-\widetilde{x}$, and we let $u_k^r$, $v_k^r$ be obtained
from the corresponding null solutions $\widetilde{u}_k^\ell$,
$\widetilde{v}_k^\ell$ of $\widetilde{P}_\varepsilon -E$. Note that
the Stokes sectors $\Sigma ^r_k$, $S^r_k$ correspond to the sectors
$\widetilde{\Sigma }^\ell_k$, $\widetilde{S}^\ell _k$ to the left,
defined exactly as $\Sigma _k^\ell$, $S_k^\ell$. (Cf.~Figure
\ref{sw.fig.1}.)

This means that we have the 6 solutions $u_k^r$, $v_k^r$, $k\in \Z$
which satisfy (\ref{sw.3}) with ``$\ell$'' replaced by ``$r$'' and
$\varphi _k^r$, $\psi _k^r$ are solutions to the eikonal equation
(\ref{sw.4}), vanishing at $z=\alpha _r(E,\varepsilon )$, $z=\beta
_r(E,\varepsilon )$ respectively, satisfying (\ref{sw.5}) with
``$\ell$'' replaced by ``$r$''. The principal parts $a^r_{k,0}$,
$b^k_{k,0}$ are given by (\ref{sw.8}), (\ref{sw.10}), (\ref{sw.11}),
(\ref{sw.14}), (\ref{sw.16}), (\ref{sw.17}) with ``$\ell$" replaced by
``$r$''. Here $(V_\varepsilon (x)-E)^{\frac{1}{4}}_r$ is the branch
which is $\ge 0$ to the right of $\alpha_{r}$, when $\varepsilon =0$
and $E$ is real and close to $E_0$. Again, we can modify the choice of
$u_\pm^r$ by constant factors $1+{\cal O}(h)$, so that the analogue of
(\ref{sw.12}) holds:
\begin{equation}\label{sw.22}
u_0^r=u_{-1}^r+u_1^r.
\end{equation}
Then we can modify $v_{\pm 1}^r$ by constant factors $1+{\cal O}(h)$
so that (\ref{sw.21}) holds with ``$\ell$'' replaced by ``$r$'':
\begin{equation}\label{sw.23}
\begin{cases}
u_{-1}^r=-ie^{iI_{r}(E,\varepsilon )/h}v_1^r,\\
u_{1}^r=ie^{-iI_{r}(E,\varepsilon )/h}v_{-1}^r,
\end{cases}
\end{equation}
where now (cf.~(\ref{sw.18})
\begin{equation}\label{sw.24}
I_{r}(E,\varepsilon )=\int _{\beta_{r}}^{\alpha_{r}} (E-V_\varepsilon (t))_w^{\frac{1}{2}}dt
\end{equation}
and $(E-V_\varepsilon )_w^{\frac{1}{2}}$ is now defined near the right
well $]\beta _r,\alpha _r[$ as the branch of the square root which is
positive on $]\beta _r,\alpha _r[$ when $\varepsilon =0$ and $E$ is
real and close to $E_0$. 

Notice that this fits with the principle of
transforming everything from the right to the left by putting
$\widetilde{V}_\varepsilon (\widetilde{x})=V_\varepsilon
(-\widetilde{x})$. $\widetilde{V}_\varepsilon $ has the left turning
points $\widetilde{\alpha }_\ell=-\alpha _r$, $\widetilde{\beta }_\ell
=-\beta _r$ and
$$
\int_{\widetilde{\alpha }_\ell}^{\widetilde{\beta
  }_\ell}(E-V_\varepsilon (\widetilde{x}))^{\frac{1}{2}}
d\widetilde{x}=
-\int_{\alpha _r}^{\beta _r}(E-V_{\varepsilon }(x))^{\frac{1}{2}} dx=I_r,
$$ 
with the natural branches of the square root.

Now we build two convenient independent formal WKB solutions
$w_{0}^\ell$ and $w_{0}^r$ near the barrier.  We set
\begin{equation}\label{sw.25}
w_{0}^\ell=\frac{1}{2i}(v^\ell_{1}-v^\ell_{-1}).
\end{equation}
Since
$$
v_{1}^\ell 
= i(V_{\varepsilon}-E)_{m}^{-\frac{1}{4}}e^{i\psi^\ell_{1}/h}(1+\CO(h))
= i(V_{\varepsilon}-E)_{m}^{-\frac{1}{4}}e^{-i\psi^\ell_{0}/h}(1+\CO(h)),
$$
and
$$
v_{-1}^\ell 
= -i(V_{\varepsilon}-E)_{m}^{-\frac{1}{4}}e^{i\psi^\ell_{1}/h}(1+\CO(h))
= -i(V_{\varepsilon}-E)_{m}^{-\frac{1}{4}}e^{-i\psi^\ell_{0}/h}(1+\CO(h)),
$$
we have
$$
w_{0}^\ell = (V_{\varepsilon}-E)_{m}^{-\frac{1}{4}}e^{-i\psi^\ell_{0}/h}(1+\CO(h)).
$$
On the other hand, we have
$$
v_{0}^r= (V_{\varepsilon}-E)_{m}^{-\frac{1}{4}}e^{i\psi^r_{0}/h}(1+\CO(h)),
$$
so that
\begin{equation*}
w_{0}^\ell =\delta_{\ell,r}e^{J(E,\varepsilon)/h} v_{0}^r,
\end{equation*} 
where  $\delta_{\ell,r}=1+\CO(h)$ is a symbol, and
\begin{equation}
\label{sw.26}
J(E,\varepsilon)=
-i\psi_{0}^r -i\psi_{0}^\ell= \int_{\beta_{\ell}}^{\beta_{r}}(V_{\varepsilon}(t)-E)^{\frac{1}{2}}dt.
\end{equation}
We now replace $v_0^r$ by $\delta _{\ell ,r}v_0^r$ (which does not
modify the leading asymptotics) so that
\begin{equation}\label{sw.27}
w_{0}^\ell =e^{J(E,\varepsilon)/h} v_{0}^r,
\end{equation} 
Notice that (\ref{sw.27}) fixes a choice for the formal WKB solution $v_{0}^r$.

In the same way, we set
\begin{equation}\label{sw.28}
w_{0}^r=\frac{1}{2i}(v^r_{1}-v^r_{-1}).
\end{equation}
and we have
$$
w_{0}^r=\delta_{r,\ell}e^{J(E,\varepsilon)/h} v_{0}^\ell,
$$
then replace $v_0^\ell $ by $\delta _{r,\ell}v_0^\ell$ and get
\begin{equation}\label{sw.29}
w_0^r=e^{J(E,\varepsilon )/h}v_0^\ell,
\end{equation}
which we take as the definition of $v_0^\ell$.

\par In analogy with the equation prior to (\ref{sw.11}) we have
\begin{equation}\label{sw.30}
v_0^\ell=\gamma _+^\ell (h)v_1^\ell +\gamma _-^\ell (h)v_{-1}^\ell ,
\end{equation}
\begin{equation}\label{sw.31}
v_0^r=\gamma _+^r (h)v_{-1}^r +\gamma _-^r (h)v_1^r ,
\end{equation}
where $\gamma _\pm^{\bullet }=1+{\cal O}(h)$. Having already adjusted
$v_\pm^\bullet $ by factors $1+{\cal O}(h)$, there is no place for
further adjustments, so we have to refrain from the possibility of replacing
$\gamma _\pm^\bullet $ by $1$.

\section{WKB expansions of $L^2$ solutions outside the well}
\label{si}

In this section we focus on solutions of the Schr\"odinger equation (\ref{sw.1}) that are $L^2$ in a neighborhood of $+\infty$ or $-\infty$ respectively. The existence of such solutions follows from the general theory of partial differential equations since $P_{h,\varepsilon}-E$ is elliptic at infinity for $E\in D(E_{0},\varepsilon_{0})$. We are interested in their asymptotic behavior  as $h\to 0$.

Let $\delta _{0}>0$. We consider the eikonal (\ref{wkb.5}) and transport equations (\ref{wkb.6}) on the half-line $]\alpha_{r}+\delta _{0},+\infty[$.
There exists $\varepsilon _{0}=\varepsilon _{0}(\delta _{0})$
small enough, such that for all $(E,\varepsilon) \in D(E_{0},\varepsilon _{0})\times D(0,\varepsilon_{0})$,
and for  $x\in ]\alpha_{r}+\delta _{0},+\infty[$, the function
\begin{equation}
\label{si.1}
\varphi_{+} (x,E)=i\int_{\alpha_{r}(E,\varepsilon)}^{x}(V_{\varepsilon }(t)-E)_{r}^{\frac{1}{2}}dt,
\end{equation}
is a smooth function of $x$, and an analytic function of $(E,\varepsilon)$, which solves the eikonal equation. 
We recall that  $t\mapsto (V_{\varepsilon }(t)-E)_{r}^{\frac{1}{2}}$ is real and positive in $]\alpha_{r}+\delta _{0},+\infty[$ when $\varepsilon=0$, so that 
\begin{equation}
\label{si.2}
\re (i\varphi_{+}(x,E))<0.
\end{equation}
It is then straightforward to obtain the existence of the solutions of the corresponding transport equations, and we get the

\begin{proposition}\sl
\label{si.3}
For all $\delta _{0}>0$, there exists $\varepsilon _{0}>0$ such that, for all $(E,\varepsilon )\in
D(E_{0},\varepsilon _{0})\times D(0,\varepsilon _{0})$, the equation (\ref{sw.1}) has a formal WKB solution $\tilde u_{+}$ in
$]\alpha_{r}+\delta _{0},+\infty [$,
\begin{equation}\label{si.4}
\tilde u_{+}(x,\varepsilon ,E,h)=e^{i\varphi_{+} (x,E)/h}\sum_{j\geq 0}a_{j}(x,E,\varepsilon)h^j,
\end{equation}
where $a_{j}$ is $C^{\infty}$ with respect to $x\in ]\alpha_{2}+\delta _{0},+\infty[$, and is holomorphic with respect to $(E,\varepsilon)\in D(E_{0},\varepsilon _{0})\times D(0,\varepsilon _{0})$. Moreover we can choose $a_{0}$ so that,
\begin{equation}
a_{0}(x,E,\varepsilon)=(-i\partial_{x}\varphi_{+})^{-\frac{1}{2}},
\label{si.5}
\end{equation}
and  we have the estimates, for all $j\geq 0$ and all $k\geq 0$,
\begin{equation}
\label{si.6}
\vert a^{(k)}_{j}(x,E,\varepsilon)\vert=\CO(\< x\>^{-\tfrac{m_{0}}{4}-j(\tfrac{
m_{0}}{2}+1{si.5})-k}).
\end{equation}
\end{proposition}

\begin{proof}
We prove (\ref{si.6}).  Differentiating the eikonal equation
$$
\varphi'_{+}(x)^2= E-V_{\varepsilon}(x),
$$
and using (A2), we easily obtain by induction that, for $k\geq 1$,
\begin{equation}\label{si.7}
(\varphi'_{+})^{(k)}(x)=\CO(\<x\>^{\frac{m_{0}}{2}-k}).
\end{equation}
Of course we also have, by (A4),
$$
\<x\>^{\frac{m_{0}}{2}}\lesssim \vert\varphi'_{+}(x)\vert,
$$
so that (\ref{si.5}) gives
\begin{equation}\label{si.8}
\< x\>^{-\frac{m_{0}}{4}}\lesssim \vert a_{0}(x,E,\varepsilon)\vert  \lesssim \< x\>^{-\frac{m_{0}}{4}}.
\end{equation}
Then, differentiating the first transport equation
\begin{equation}\label{si.9}
\varphi_{+}'(x)a_{0}'(x,E,\varepsilon)+\frac{\varphi''_{+}(x)}{2}a_{0}(x,E,\varepsilon)=0,
\end{equation}
we obtain by induction that, uniformly for $(E,\varepsilon)\in D(E_{0},\varepsilon_{0})\times D(0,\varepsilon_{0})$,
\begin{equation}\label{si.10}
\vert a^{(k)}_{0}(x,E,\varepsilon)\vert\lesssim \< x\>^{-m_{0}/4-k},
\end{equation}
which proves the estimates for $j=0$ and all $k\in \N$.
Now suppose that (\ref{si.6}) holds for some $j$ and all $k\in \N$.
The $(j+1)$-st transport equation is,
\begin{equation*}
(\varphi_{+}'(x) \partial _{x}+\frac{i}{2}{\varphi_{+}''(x)}
)a_{j+1}=\frac{i}{2}a''_{j},
\end{equation*}
and, setting
\begin{equation}\label{si.11}
a_{j+1}=f_{j+1}a_{0},
\end{equation}
we  get, using  also (\ref{si.9}), 
\begin{equation}\label{si.12}
\big (a_{0}(x)\varphi_{+}' (x)\big )f_{j+1}'(x)=\frac{i}{2}a''_{j}(x).
\end{equation}
Thus we have first
\begin{equation}\label{si.13}
f_{j+1}(x) =
\frac{i}{2}\int_{+\infty }^{x}\frac{a''_{j}(t)}{\varphi_{+}' (t)a_{0}(t)}dt=
\CO(\< x\>^{-(j+1)(\frac{m_{0}}{2}+1)}),
\end{equation}
and, differentiating (\ref{si.12}), we obtain by induction
\begin{equation}\label{si.14}
f_{j+1}^{(k)}(x) =
\CO(\< x\>^{-(j+1)(\frac{m_{0}}{2}+1)-k}).
\end{equation}
Then (\ref{si.6}) follows by differentiating (\ref{si.11}) and using Leibniz formula. 
\end{proof}

To the formal series $\sum_{j\geq 0} a_{j}h^j$ defined in Proposition \ref{si.3}, we can associate a function $a$ by means of a Borel construction, setting
\begin{equation}\label{si.15}
a(x,E,\varepsilon,h)=\sum_{j\geq 0} a_{j}(x,E,\varepsilon)h^j\chi(\lambda_{j}h)
\end{equation}
for some  plateau function $\chi\in \CC^\infty_{0}(\R)$ over $\{0\}$, and a suitable sequence $(\lambda_{j})$ of   real numbers such that $\lambda_{j}\to +\infty$ as $j\to +\infty$ (see e.g. \cite[Chapter 2]{DiSj99}).
Then, 
for all $N\in N$  and any $k\geq 0$,
\begin{equation}\label{si.16}
\vert a^{(k)}(x,h)- \sum_{j=0}^{N-1} a_{j}^{(k)}(x)h^j\vert  =\CO (h^{N}\<x\>^{-\frac{m_{0}}{4}-N(\frac{m_{0}}{2}+1)-k }).
\end{equation}
Moreover $u_{+,wkb}=e^{i\varphi_{+}/h}a$ is an approximate solution to the Schr\"odinger equation (\ref{sw.1}), in the sense that
\begin{equation}\label{si.17}
P_{h,\varepsilon}u_{+,wkb}(x,E,\varepsilon,h)=e^{i\varphi_{+} (x,E,\varepsilon)/h}r(x,E,\varepsilon,h),
\end{equation}
where, for all $N\in \N$,  
\begin{equation}\label{si.18}
r(x,E,\varepsilon,h)=\CO(h^{N}\<x\>^{-N}).
\end{equation}

Now we  build a solution $u_{+}$ that has the formal WKB solution constructed above as an asymptotic expansion in $]\alpha_{r}+\delta_{0},+\infty[$.   To do so, we establish first some estimates for the solutions of the inhomogeneous Schr\"odinger equation 
\begin{equation}\label{si.19}
-h^2u''+(V_{\varepsilon}-E)u =v,
\end{equation}
on intervals of the form $\ds I_{\lambda}=[\frac{\lambda}{2}, \frac{3\lambda}{2}]$, for large $\lambda$.  For simpler notations we write 
\begin{equation}
Q(x)= Q_{E,\varepsilon}(x)=V_{\varepsilon}-E,
\end{equation}
and we set
\begin{equation}
x=\lambda + \lambda \tilde x,
\end{equation}
so that $\tilde x \in [-\frac12,\frac 12]$. Multiplying also by $\lambda^{-m_{0}}$, the equation (\ref{si.19}) on $I_{\lambda}$ is equivalent to the equation
\begin{equation}\label{si.20}
(\tilde h^2D_{\tilde x}^2 +\tilde Q(\tilde x))\tilde u=\tilde v,
\end{equation}
on $[-\frac12,\frac 12]$, where
\begin{equation}
\tilde h=\frac{h}{\lambda^{1+m_{0}/2}}, \ \tilde Q(\tilde x)=\lambda^{-m_{0}} Q(x), \ \tilde v(\tilde x)=\lambda^{-m_{0}}v(x).
\end{equation}
Notice in particular that $\tilde Q\asymp 1$, and that $\tilde Q^{(k)} =\CO(1)$ for all $k\geq 1$. As in \cite[Chapter 7]{Sj14+}, we write (\ref{si.20}) as the first order system,
\begin{equation}
(\tilde hD_{\tilde x}+A(\tilde x))U=V,\   
A(\tilde x)=\begin{pmatrix}0&-1\\ \tilde Q(\tilde x)&0\end{pmatrix},\ 
U=\begin{pmatrix}\tilde u\\\tilde hD_{\tilde x}\tilde u\end{pmatrix},\  
V=\begin{pmatrix}0\\ \tilde v\end{pmatrix},
\end{equation}
and we denote $\tilde \CE=(\tilde\CE_{(i,j)})\in \CC^\infty(I\times I,\CM_{2}(\C))$ the fundamental solution of this system. One can prove that (see \cite[Theorem 7.1.3]{Sj14+}), for any $j,k\in N$,
\begin{equation}
\Vert (h\partial_{\tilde x})^j(h\partial_{\tilde x})^k\tilde \CE(\tilde x,\tilde y)\Vert \leq C_{j,k}\exp(\frac{1}{\tilde h}\vert \im \tilde \varphi(\tilde x)-\im \tilde \varphi(\tilde y)\vert),
\end{equation}
for some $C_{j,k}>0$, where $\tilde \varphi$ is the solution of the eikonal equation associated to (\ref{si.20}) such that $e^{i\tilde\varphi(\tilde x)/h}$ is decaying as $\tilde x$ increases. 
Now the solution $\tilde u$ of (\ref{si.20}) satisfies, for any $\tilde x,\tilde y\in [-\frac12,\frac 12]$,
\begin{equation}
\begin{pmatrix}
\tilde u(\tilde x)\\
\tilde hD_{\tilde x}u(\tilde x)
\end{pmatrix}
=
\tilde \CE(\tilde x,\tilde y)
\begin{pmatrix}
\tilde u(\tilde y)\\
\tilde hD_{\tilde x}u(\tilde y)
\end{pmatrix}.
\end{equation}
Since $i\tilde\varphi(\tilde x)/h=i\varphi(x)/h$, where $\varphi$ is the solution of the eikonal equation associated to (\ref{si.19}) which decays when $x$ increases, we also have, with $u(x)=\tilde u (\tilde x)$,
\begin{equation}
\begin{pmatrix}
u(x)\\
\frac{h}{\lambda^{m_{0}}}D_{x}u(x)
\end{pmatrix}
=
\tilde \CE(\tilde x,\tilde y)
\begin{pmatrix}
u(y)\\
\frac{h}{\lambda^{m_{0}}}D_{y}u(y)
\end{pmatrix},
\end{equation}
with
\begin{equation}\label{si.21}
\Vert \tilde \CE(\tilde x,\tilde y)\Vert 
\leq
C\exp({\frac{1}{h}\vert\im\varphi(x)-\im\varphi(y)\vert}.
\end{equation}

Replacing $\lambda$ by $\vert x\vert$ we get
\begin{equation}
\begin{pmatrix}
u(x)\\
hDu(x)
\end{pmatrix}
=
\CE(x, y)
\begin{pmatrix}
u(y)\\
 hDu(y)
\end{pmatrix},
\end{equation}
where
\begin{equation}
\CE(x,y)=
\begin{pmatrix}
1&0\\
0&\vert x \vert^{m_{0}/2}
\end{pmatrix}
\hat \CE(x,y)
\begin{pmatrix}
1&0\\
0&\vert y\vert^{-m_{0}/2}
\end{pmatrix},
\end{equation}
and $\hat\CE$ satisfies the estimate (\ref{si.21}) possibly with another constant $C>0$.

For general $\alpha_{r}+\delta_{0}\leq x\leq y$, we can cover $[x,y]$ with $1+\CO(\ln(\frac{y}{x}))$ intervals of the type $I_{\lambda}$. Thus, writing
$$
\CE(x,y)=\CE(x,x_{1})\CE(x_{1},x_{2})\dots \CE(x_{n},y),
$$
where each of the intervals $[x_{j},x_{j+1}]$ is contained in one of the chosen $I_{\lambda}$'s, we obtain
\begin{equation}
\Vert \CE(x,y)\Vert \leq C^{1+\CO(\ln(\frac{y}{x}))}\exp(\frac{1}{h}\vert \im \varphi(x)-\im\varphi(y)\vert y^{-m_{0}/2}.
\end{equation}
In particular we have, for all $\alpha_{r}+\delta_{0}\leq x\leq y$,
\begin{equation}\label{si.22}
\vert \CE_{12}(x,y)\vert =\frac{\CO(1)}{h} \frac{y^{C-m_{0}/2}}{x^C} \exp(\frac{1}{h} \im (\varphi(y)-\varphi(x))),
\end{equation}
where we have used the fact that for $x\leq y$,  $\im \varphi(y)\geq \im \varphi(x)$.

Now we consider again the function $u_{+,wkb}$, and we denote $u_{\lambda}$ the unique solution of the Schr\"odinger equation (\ref{sw.1}) such that
\begin{equation}\label{si.23}
\left\{
\begin{array}
{l}
u_{\lambda}(\lambda,E,\varepsilon,h)=u_{+,wkb}(\lambda,E,\varepsilon,h),\\[6pt]
hDu_{\lambda}(\lambda,E,\varepsilon,h)=hDu_{+,wkb}(\lambda,E,\varepsilon,h).
\end{array}
\right .
\end{equation}
We have
\begin{equation}
P_{h,\varepsilon}(1_{]\alpha_{r}+\delta_{0},\lambda]}(u_{\lambda}-u_{+,wkb}))=
1_{]\alpha_{r}+\delta_{0},\lambda]}r,
\end{equation}
so that, for $x\in ]\alpha_{r}+\delta_{0},\lambda]$,
\begin{equation}
u_{\lambda}(x)-u_{+,wkb}(x)=-\frac{i}{h}\int_{x}^\lambda \CE_{(1,2)}(x,y) r(y) dy.
\end{equation}
Then, noticing that we can take $C>0$ arbitrarily large in (\ref{si.22}) , we get, for all $N\in \N$ and all $k\in \N$,
\begin{equation}\label{si.24}
u^{(k)}_{\lambda}(x)-u^{(k)}_{+,wkb}(x)=\CO(h^N\<x\>^{-N} e^{-\im \varphi (x)/h}),
\end{equation}
and, for $\lambda_{1}<\lambda_{2}\in \R^+$ large enough,
\begin{equation}
u^{(k)}_{\lambda_{2}}(x)-u^{(k)}_{\lambda_{1}}(x)=\CO(h^N\<x\>^{-N} \lambda_{1}^{-N}e^{-\im \varphi (x)/h}),
\end{equation}
Thus, the family $(u_{\lambda})$ converges to some function $u_{+}$, which is an exact solution to (\ref{sw.1}), and (\ref{si.24}) gives that, for all $N\in \N$ and all $k\in \N$,
\begin{equation}
u^{(k)}_{+}(x)-u^{(k)}_{+,wkb}(x)=\CO(h^N\<x\>^{-N} e^{-\im \varphi(x)/h}).
\end{equation}

We have proved the main part of the 

\begin{proposition}\sl
\label{si.25}
Let $\delta _{0}>0$, and $\varepsilon_{0}>0$ be small enough. Let $\tilde u=e^{i\varphi_{+}/h}\sum_{j\geq 0}a_{j}h^{j}$ be a formal WKB solution in $]\alpha_{r}+\delta_{0},+\infty[$ satisfying (\ref{si.1}) and (\ref{si.5}). Then,
for any $(E,\varepsilon )\in D(E_{0},\varepsilon _{0})\times D(0,\varepsilon_{0})$, the equation (\ref{sw.1}) has a unique solution  $u_{+}$ on  $]\alpha_{r}+\delta _{0},+\infty [$  such that,
\begin{equation}\label{si.26}
u_{+}(x,E,\varepsilon,h)=a_{+}(x,E,\varepsilon,h)e^{i\varphi_{+} (x,E,\varepsilon)/h}
\end{equation}
with 
\begin{equation}
a_{+}(x,E,\varepsilon,h)\sim
\sum_{j=0}^{\infty }a^+_{j}(x,E,\varepsilon)h^{j},
\label{si.27}
\end{equation}
in the sense that
\begin{align}
\nonumber
\forall N\in \N^*, &\forall k \in \N, \exists C_{N,k}>0
\mbox{ such that}
\\
&\vert \partial_{x}^k\Big (a_{+}(x,h)-\sum_{j=0}^{N-1}a_{j}^+(x)h^{j}\Big)\vert
\leq
C_{N,k}h^{N}\<x\>^{-\frac{m_{0}}{4}-N(\frac{m_{0}}{2}+1)-k }.
\label{si.28}
\end{align}
Moreover $u_{+}$ belongs to $L^2(]\alpha_{r}+\delta_{0},+\infty[$, and it is analytic with respect to $(E,\varepsilon)\in D(E_{0},\varepsilon _{0})\times D(0,\varepsilon_{0})$. 
\end{proposition}

\begin{proof}  It only remains to prove that  the solution $u_{+}$ belongs to $L^2(]\alpha_{r}+\delta_{0},+\infty[)$. Assumption (A2) implies that there exists $C>0$ such that $\re(V_{\varepsilon}(x)-E)>\frac{1}{C^2}$ for all  $x$ large enough. Thus we have, for $x\in ]\alpha_{2},+\infty[$ large enough,
\begin{equation}\label{si.29}
\re(i\varphi_{+}(x))=-\re\int_{\alpha_{r}}^x (V_{\varepsilon}(t)-E)^{\frac{1}{2}}dt\leq -\frac{x}{C} +C.
\end{equation}
On the other hand, the estimate (\ref{si.28}) for $N=0$, $k=0$,  gives
$a_{+}(x,E,\varepsilon)=\CO(\<x\>^{-m_{0}/4})$, so that $u_{+}\in L^2(]\alpha_{r}+\delta_{0},+\infty[)$.
\end{proof}

It is clear that we have the same result for the existence of a
solution $u_{-}\in L^2(\R^-)$ that has, for any $\delta_{0}>0$, a WKB
asymptotic expansion in $]-\infty,\alpha _{\ell}-\delta _{0}[$ of the form
\begin{equation}\label{si.30}
u_{-}(x,E,\varepsilon,h)=a_{-}(x,E,\varepsilon,h)e^{i\varphi_{-} (x,E,\varepsilon)/h},
\end{equation}
where the phase $\varphi_{-}$ is defined by
\begin{equation}\label{si.31}
i\varphi_{-} (x,E,\varepsilon)=\int_{\alpha_{\ell}(E,\varepsilon)}^x (V_{\varepsilon}(t)-E)_{\ell}^{\frac{1}{2}} dt
\end{equation}
where we recall that the determination of $t\mapsto (V_{\varepsilon}(t)-E)_{\ell}^{\frac{1}{2}}$ is fixed in such a way that, for $x\in ]-\infty,\alpha_{\ell}-\delta_{0}[$,
\begin{equation}\label{si.32}
\re(i\varphi_{-} (x,E,\varepsilon))<0.
\end{equation}
We also have, in the same sense as in (\ref{si.28}),
\begin{equation}\label{si.33}
a_{-}(x,E,\varepsilon,h)\sim
\sum_{j=0}^{\infty }a_{j}^-(x,E,\varepsilon)h^{j},
\end{equation}
where $u_{wkb}=e^{i\varphi_{-}/h}\sum_{j=0}^{\infty }a_{j}^-h^{j}$ is a formal WKB solution.
It is of course also analytic with respect to $(E,\varepsilon)\in D(E_{0},\varepsilon _{0})\times D(0,\varepsilon_{0})$.

\section{The quantization condition}\label{qc}

\par To start with, we derive the quantization condition, using only the double
well structure but not yet the ${\cal PT}$-symmetry nor the symmetry
following from the fact that $P_0$ is real and self-adjoint. 

In Section \ref{si} we have introduced the two null solutions $u_+$,
$u_-$ of $(P_\varepsilon -E)u=0$ that decay exponentially near $+\infty $
and $-\infty $ respectively, and we know that $E$ is an eigenvalue of
$P_\varepsilon $ precisely when $W(u_+,u_-)=0$ or equivalently when $u_+$
and $u_-$ are colinear. It is clear that we can choose $u_\pm$ and
$u_0^\ell$, $u_0^r$ of the preceding section, so that
$$
u_0^\ell=u_-,\ \ u_0^r=u_+.
$$

\par By (\ref{sw.12}), (\ref{sw.21}) we have
\begin{equation}\label{qc.1}
u_0^\ell=-ie^{iI_{\ell} /h}v_1^\ell +ie^{-iI_{\ell} /h}v_{-1}^\ell .
\end{equation}
Similarly, by (\ref{sw.22}), (\ref{sw.23}),
\begin{equation}\label{qc.2}
u_0^r=-ie^{iI_{r}/h}v_1^r+ie^{-iI_{r}/h}v_{-1}^r.
\end{equation}

\par Here, we recall (\ref{sw.25}), (\ref{sw.27}), implying
\begin{equation}\label{qc.3}
v_0^r=e^{-J/h}\frac{1}{2i}\left(v_1^\ell -v_{-1}^\ell \right)
\end{equation}
and (\ref{sw.28}), (\ref{sw.29}), that give
\begin{equation}\label{qc.4}
v_0^\ell =e^{-J/h}\frac{1}{2i}\left(v_1^r -v_{-1}^r \right) .
\end{equation}
(\ref{sw.30}) and (\ref{qc.3}) form a system that allows to express
$v_{\pm 1}^\ell$ in terms of $v_0^r$, $v_0^\ell$. Similarly,
(\ref{sw.31}) and (\ref{qc.4}) allow us to express $v_\pm^r$ in terms
of $v_0^r$, $v_0^\ell$. After some straightforward calculations, we
get,
\begin{equation}\label{qc.5}
\begin{pmatrix}v_1^\ell\\ v_{-1}^\ell\end{pmatrix}=\frac{1}{\gamma
  _-^\ell +\gamma _+^\ell}\begin{pmatrix}1 & 2ie^{J/h}\gamma _-^\ell
  \\ 1 & -2ie^{J/h}\gamma
  _+^\ell\end{pmatrix}\begin{pmatrix}v_0^\ell\\ v_0^r\end{pmatrix},
\end{equation}
\begin{equation}\label{qc.6}
\begin{pmatrix}v_1^r\\ v_{-1}^r\end{pmatrix}=\frac{1}{\gamma
  _-^r +\gamma _+^r}\begin{pmatrix} 2ie^{J/h}\gamma _+^r &1
  \\ -2ie^{J/h}\gamma _-^r & 1
  \end{pmatrix}\begin{pmatrix}v_0^\ell\\ v_0^r\end{pmatrix} .
\end{equation}          

\par Combining (\ref{qc.1}) and (\ref{qc.5}), we get after a straightforward calculation,
\begin{equation}\label{qc.7}
u_0^\ell=\frac{1}{\gamma _-^\ell+\gamma _+^\ell} 
\begin{pmatrix}
\frac{1}{i}\left( e^{iI_{\ell}/h}-e^{-iI_{\ell} /h} \right) &
2e^{J/h}\left(e^{iI_{\ell}/h}\gamma _-^\ell +e^{-iI_{\ell} /h}\gamma _+^\ell \right)
\end{pmatrix} 
\begin{pmatrix}
v_0^\ell \\ v_0^r\end{pmatrix}.
\end{equation}
Combining (\ref{qc.2}), (\ref{qc.6}), we get
\begin{equation}\label{qc.8}
u_0^r=\frac{1}{\gamma _-^r+\gamma _+^r} \begin{pmatrix}
2e^{J/h}\left(e^{iI_{r}/h}\gamma _+^r +e^{-iI_{r} /h}\gamma _-^r \right) 
&\frac{1}{i}\left( e^{iI_{r}/h}-e^{-iI_{r} /h} \right) 
\end{pmatrix} \begin{pmatrix}v_0^\ell \\ v_0^r\end{pmatrix}.
\end{equation}

\par Since $v_0^\ell$, $v_0^r$ are linearly independent, we see that
$E$ is an eigenvalue of $P_\varepsilon $ precisely when the two row
matrices in (\ref{qc.7}) and (\ref{qc.8}) are colinear or equivalently
when the determinant of the matrix, formed by these two rows, is equal
to $0$. We then get the quantization condition
\[
\begin{split}
0=& \frac{1}{i}\left(e^{iI_{\ell}/h}-e^{-iI_{\ell}/h}\right)
\frac{1}{i}\left(e^{iI_{r}/h}-e^{-iI_{r}/h}\right)\\
&-4e^{2J/h}\left(e^{iI_{\ell} /h}\gamma _-^\ell +e^{-iI_{\ell} /h}\gamma
  _+^\ell \right)
\left(e^{iI_{r} /h}\gamma _+^r +e^{-iI_{r} /h}\gamma _-^r \right),
\end{split}
\]
which we rewrite as 
\begin{equation}\label{qc.9}
f(E,\varepsilon )=0,
\end{equation}
where
\begin{equation}\label{qc.10}
\begin{split}
f(E,\varepsilon )=&\frac{1}{4}\left(e^{iI_{\ell} /h}\gamma _-^\ell +e^{-iI_{\ell} /h}\gamma
  _+^\ell \right)
\left(e^{iI_{r} /h}\gamma _+^r +e^{-iI_{r} /h}\gamma _-^r \right)\\
&-\frac{1}{4}e^{-2J/h}\sin \left(I_{\ell} /h \right) \sin\left(I_{r}/h
\right) .
\end{split}
\end{equation}

\par We shall now take into account the various symmetry
properties. For functions $u(x,E,\varepsilon )$, where $x,E$ vary in
some domains in $\mathbb{C}$ and $\varepsilon $ in some real domain, we
put,
\begin{equation}\label{qc.11}
\mathrm{Co\,}(u)(x,E,\varepsilon )=u^*(x,E,\varepsilon
)=\overline{u(\overline{x},\overline{E},-\varepsilon )},
\end{equation}
\begin{equation}\label{qc.12}
\mathrm{Pt\,}(u)(x,E,\varepsilon )=u^\dagger (x,E,\varepsilon
)=\overline{u(-\overline{x},\overline{E},\varepsilon )},
\end{equation}
so that $\mathrm{Pt}$ is equal to ${\cal PT}$ in the
introduction. Notice that $\mathrm{Pt}$ and $\mathrm{Co}$ are
idempotent anti-linear operators that commute: $\mathrm{Co\,}\circ
\mathrm{Pt}=\mathrm{Pt\,}\circ \mathrm{Co\,}$. 

\par Using only that $V_\varepsilon =V_0+i\varepsilon W$ with $V_0$,
$W$ real-valued on the real domain, we see that
$$
V_\varepsilon ^*(x)=\overline{V_{-\varepsilon
  }(\overline{x})}=V_\varepsilon (x) ,\ \ \mathrm{Pt\,}(V_\varepsilon
)(x)=\overline{V_\varepsilon (-\overline{x})}=V_0(-x)-i\varepsilon W(-x),
$$
and if we make the ${\cal PT}$-symmetry assumption (A5), we get
$\mathrm{Pt}(V_\varepsilon )=V_\varepsilon $ and hence that
\begin{equation}\label{qc.13}
(P_\varepsilon -E)\circ \mathrm{Pt\,}=\mathrm{Pt\,}\circ (P_\varepsilon -E).
\end{equation}
Without (A5), we still have
\begin{equation}\label{qc.14}
(P_\varepsilon -E)\circ \mathrm{Co\,}=\mathrm{Co\,}\circ (P_\varepsilon -E).
\end{equation}
To verify this quickly, we observe that
$\mathrm{Co\,}(uv)=\mathrm{Co\,}(u)\mathrm{Co\,}(v)$ for products of
functions and similarly for $\mathrm{Pt\,}$, and that
$\mathrm{Pt\,}\circ \partial _x=-\partial _{\overline{x}}\circ \mathrm{Pt\,}$,
$\mathrm{\,Co}\circ \partial _x=\partial _{\overline{x}}\circ
\mathrm{Co\,}$. Recall that to leading order, 
$$
u_0^\ell \equiv u_{0,0}^\ell :=(V_\varepsilon
-E)_{\ell}^{-\frac{1}{4}}e^{\frac{1}{h}\int_{\alpha _\ell}^x
  (V_\varepsilon -E)_\ell^{\frac{1}{2}}dt}.
$$
By straightforward calculations, oberving that $\alpha _\ell
(\overline{E},-\varepsilon )=\overline{\alpha _\ell (E,\varepsilon
  )}$ (i.e. $\alpha _\ell^*=\alpha _\ell$), we obtain,
\begin{equation}\label{qc.15}
\mathrm{Co\,}(u^\ell_{0,0})=u^\ell_{0,0}.
\end{equation}
In view of (\ref{qc.14}), we know that $\mathrm{Co\,}(u_0^\ell )$
is a null solution of $P_\varepsilon -E$ and using also (\ref{qc.15}),
we conclude that $\frac{1}{2}\left(u_0^\ell +\mathrm{Co\,}(u_0^\ell)
\right)$ is a null solution with leading asymptotics $u_{0,0}^\ell$
which is invariant under $\mathrm{Co\,}$, so if we replace $u_0^\ell$
by this function we gain the property,
\begin{equation}\label{qc.16}
\mathrm{Co\,}(u_0^\ell )=u_0^\ell .
\end{equation}

\par Similarly, in the discussion leading to (\ref{sw.11.5}) we see
that we can choose $u_{\pm 1}^\ell$ so that
\begin{equation}\label{qc.17}
u_1^\ell =\mathrm{Co\,}(u_{-1}^\ell ).
\end{equation}
Since $u_{-1}^\ell$, $u_1^\ell$ form a basis for the space of null
solutions of $P_\varepsilon -E$, we get from (\ref{sw.11.5}),
(\ref{qc.15}), that $\mathrm{Co\,}(\tau _+)=\tau _-$, so after
replacing $u_{\mp 1}^\ell$ by $\tau _{\pm}u_{\mp 1}^\ell$, we still
have (\ref{qc.17}) for the new functions $u_{\mp 1}^\ell$ in
(\ref{sw.12}).

\par Next, notice that $I_\ell$ in (\ref{sw.18}) satisfies
\begin{equation}\label{qc.18}
I_\ell^*=I_\ell .
\end{equation}
This means that
\begin{equation}\label{qc.19}
v_1^\ell =\mathrm{Co\,}(v_{-1}^\ell ),
\end{equation}
for $v_{\pm 1}^\ell$ in (\ref{sw.21}).

\par The whole discussion so far applies with ``$\ell$'' replaced by
``$r$'' and we get
\begin{equation}\label{qc.20}
\mathrm{Co\,}(u_0^r)=u_0^r,\ I_r^*=I_r,\
u_1^r=\mathrm{Co\,}(u_{-1}^r),\ v_1^r=\mathrm{Co\,}(v_{-1}^r).
\end{equation}
Moreover, we check that 
\begin{equation}\label{qc.21}
J^*=J.
\end{equation}

\par From (\ref{sw.25}), (\ref{sw.27}) we now get
\begin{equation}\label{qc.22}
\mathrm{Co\,}(v_0^r)=v_0^r.
\end{equation}
Similarly,
\begin{equation}\label{qc.23}
\mathrm{Co\,}(v_0^\ell )=v_0^\ell .
\end{equation}
Then in (\ref{sw.30}), (\ref{sw.31}) we must have 
\begin{equation}\label{qc.24}
\mathrm{Co\,}(\gamma _+^\ell )=\gamma _-^\ell,\ \ \mathrm{Co\,}(\gamma
_+^r)=\gamma _-^r.
\end{equation}
It follows that $f$ in (\ref{qc.9}), (\ref{qc.10}) satisfies
\begin{equation}\label{qc.25}
f^*=f.
\end{equation}

\par Let us finally use the ${\cal PT}$-symmetry assumption (A5) or
equivalently (\ref{qc.13}). We then check that
\begin{equation}\label{qc.26}
I_\ell^\dagger =I_r,
\end{equation}
and
\begin{equation}\label{qc.27}
J^\dagger=J,
\end{equation}
and that we can choose 
\begin{equation}\label{qc.28}
u_0^r=\mathrm{Pt\,}(u_0^\ell),\ u_1^r=\mathrm{Pt\,}(u_{-1}^\ell ),\
u_{-1}^r=\mathrm{Pt\,}(u_1^\ell ),
\end{equation}
\begin{equation}\label{qc.29}
v_{\pm 1}^r=\mathrm{Pt\,}(v_{\mp 1}^\ell ).
\end{equation}
Then by (\ref{qc.3}), (\ref{qc.4}),
\begin{equation}\label{qc.30}
v_0^r=\mathrm{Pt\,}(v_0^\ell ),
\end{equation}
and from (\ref{sw.30}), (\ref{sw.31}) we infer that 
\begin{equation}\label{qc.31}
\gamma _{\pm}^r=\mathrm{Pt\,}(\gamma _{\pm}^\ell ).
\end{equation}
It follows that 
\begin{equation}\label{qc.32}
f^\dagger =f.
\end{equation}

\par We have seen that the zeros of $f(\cdot ,\varepsilon )$ coincide
with the eigenvalues of $P_\varepsilon $ in a neighborhood of
$E_0$. We end this section by showing that the multiplicities agree
also. 

\par Recall (\ref{qc.7}), (\ref{qc.8}) that we write as 
\[
\begin{split}
u_0^\ell&=a^\ell v_0^\ell +b^\ell v_0^r,\\
u_0^r&=a^r v_0^\ell+ b^r v_0^r.
\end{split}
\]
Taking the Wronskians, we get
$$
W(u_0^\ell,u_0^r)=\det \begin{pmatrix}a^\ell & b^\ell\\ a^r
  &b^r\end{pmatrix} W(v_0^\ell ,v_0^r). 
$$
Here the last Wronskian is non-vanishing, so up to a non-vanishing
holomorphic factor $f$ is equal to $W(u_0^\ell ,u_0^r)$ and the zeros
of $f$, counted with their multiplicity coincide with those of
$W(v_0^\ell,v_0^r)$. 

\par Hence it remains to identify eigenvalues of $P_\varepsilon $ counted
with their multiplicity with the zeros of the Wronskian $W(u_0^\ell,
u_0^r)$. For that we can widen the perspective slightly and apply a
general discussion: 

\par Let $a\in \mathbb{R}$. Let $\lambda $ vary in
$\mathrm{neigh\,}(E_0,\mathbb{C})$. (The symbol ``$E$'' will
temporarily be used to denote operators.) Using the ellipticity of
$P_\varepsilon -\lambda $ near $+\infty $, we see that
the right Dirichlet problem 
$$
(P_\varepsilon -\lambda )u=v,\ \ u(a)=v_+
$$
has a unique solution $u\in H^2(]a,+\infty [)$ for every $(v,v_+)\in
H^0(]a,+\infty [)\times \mathbb{C}$. Similarly, by using the
ellipticity near $-\infty $ we see that the corresponding left
Dirichlet problem has a unique solution $u\in H^2(]-\infty ,a [)$ for every $(v,v_+)\in
H^0(]-\infty ,a [)\times \mathbb{C}$. Denote the solutions to the two
problems by $u=E_rv+E_r^+v_+$ and $u=E_\ell v+E_\ell ^+v_+$ respectively. 

\par It follows that the Grushin problem 
$$
\begin{cases}
(P_\varepsilon -\lambda )u+R_-u_-=v,\\
R_+u=v_+,
\end{cases}
$$
has a unique solution 
$$
(u,u_-)\in \left(\left(H^2(]-\infty ,a[)\oplus H^2( ]a,+\infty [)
  \right)\cap H^1(\mathbb{R}) \right) \times \mathbb{C}
$$
for every $(v,v_+)\in H^0(\mathbb{R})\times \mathbb{C}$, where 
$$
R_+u:=u(a),\ \ R_-u_-=u_-\delta _a,
$$
and $\delta _a$ denotes the delta function at $x=a$. Indeed, the
solution is given by
$$
u(x)=\begin{cases} E_\ell v+E_\ell^+v_+,\ x<a\\ E_r v+E_r^+v_+,\ x>a,
\end{cases} \ \ u_-=(h^2\partial _zu)(a-0)-(h^2\partial _zu)(a+0).
$$

\par We write this solution,
$$
\begin{pmatrix} u\\ u_-\end{pmatrix} =
\begin{pmatrix} E &E_+\\ E_-
  &E_{-+}\end{pmatrix} \begin{pmatrix}v\\ v_+\end{pmatrix}. 
$$
It is a standard fact  for the Grushin reduction  (see \cite[Section 6, Appendix A]{MeSj03}, or \cite{SjZw07}) that the eigenvalues
counted with their multiplicity, coincide (near $E_0$) with the zeros
of $E_{-+}(z)$ counted with their multiplicity ($\varepsilon $ is here
fixed and suppressed from the notation most of the time).  For completeness, we recall the proof. Let $z_{0}$ be an eigenvalue of $P$. Its multiplicity $m(z_{0})$ is equal to the rank, and hence to the trace of the spectral projection
$$
m(z_{0})=\tr\frac{1}{2i\pi}\int_{\gamma} (z-P)^{-1} dz,
$$
where $\gamma$ is the oriented boundary of a small disc centered at $z_{0}$. Now
$$
(z-P)^{-1}=-E(z)+E_{+}(z)E^{-1}_{-+}(z)E_{-}(z),
$$
and $E(z)$ is holomorphic near $z_{0}$, so 
\begin{equation}
\begin{split}
m(z_{0}) &= \tr\frac{1}{2i\pi}\int_{\gamma} E_{+}(z)E^{-1}_{-+}(z)E_{-}(z) dz
\\
&=\frac{1}{2i\pi}\int_{\gamma} \tr(E_{+}(z)E^{-1}_{-+}(z)E_{-}(z)) dz\\
&=\frac{1}{2i\pi}\int_{\gamma} \tr(E^{-1}_{-+}(z)E_{-}(z)E_{+}(z)) dz\\
&=\frac{1}{2i\pi}\int_{\gamma} E^{-1}_{-+}(z)E_{-}(z)E_{+}(z) dz.
\end{split}
\end{equation}
Finally, since $E_{-}(z)E_{+}(z)=\partial_{z}E^{-1}_{-+}(z)$, we see that $m(z_{0})$ is equal to the multiplicity of
$z_{0}$ as a zero of $E_{-+}$.

\par Denoting $u_-=E_\ell^+(1)$, $u_+=E_r^+(1)$, we notice that 
$$
E_{-+}=h^2(\partial _zu_-(a)-\partial _zu_+(a))=hW(u_-,u_+),
$$
since $u_\mp (a)=1$. Thus, in a neighborhood of $E_0$, the eigenvalues
of $P_\varepsilon $ counted with their multiplicity can be identified
with the zeros of $\lambda \mapsto W(u_-(\lambda ),u_+(\lambda ))$. 

\par If $\mathrm{neigh\,}(E_0,\mathbb{C})\ni \lambda \mapsto
\widetilde{u}_{\mp}(z,\lambda )$ are holomorphic families of null
solutions to $P_\varepsilon -\lambda $, exponentially decaying near $\mp
\infty $ and $\not\equiv 0$, $\forall \lambda $, then
$\widetilde{u}_\mp =\sigma _\mp (\lambda )u_\mp$, where $\sigma _\mp$
are holomorphic in $\lambda $ and non-vanishing. Thus the zeros of
$W(\widetilde{u}_-,\widetilde{u}_+)=\sigma _-(\lambda )\sigma
_+(\lambda )W(u_-,u_+)$ coincides with those of $W(u_-,u_+)$ and we have
completed the identification.

\section{The behaviour of the eigenvalues}\label{ev}

In order to study the zeros of $f(\cdot ,\varepsilon )$, we first
recollect the various symmetries:
\begin{equation}\label{ev.2}
J^*=J=J^\dagger,\ \ I_\ell^*=I_\ell,\ \ I_r^*=I_r,\ \ I_\ell^\dagger =I_r,
\end{equation}
and
\begin{equation}\label{ev.3}
\left(\gamma _-^\ell \right)^*=\gamma _+^\ell, \ \ \left(\gamma _-^r\right)^*=\gamma _+^r, 
\ \ \left(\gamma _\pm^\ell \right)^\dagger =\gamma _\pm ^r ,
\end{equation}
where we recall that $\gamma _{\pm }^{\bullet } =1+{\cal O}(h)$.

\par Let us first look at the factor
\begin{equation}\label{ev.4}
g=g_\ell (E,\varepsilon )=\left( e^{\frac{i}{h}I_\ell}\gamma _-^\ell +e^{-\frac{i}{h}I_\ell}\gamma _+^\ell \right)
\end{equation}
and drop the super/subscript $\ell $ when convenient to do so. From (\ref{ev.2}),
(\ref{ev.3}), we infer that
\begin{equation}\label{ev.5}
g^*=g
\end{equation}
Write
\begin{equation}\label{ev.6}
\begin{split}
\gamma ^\ell_-&=\left(\gamma _-^\ell\gamma _+^\ell \right)^{\frac{1}{2}}\left(\gamma
  _-^\ell/\gamma _+^\ell \right)^{\frac{1}{2}}=:\rho _\ell e^{i\theta
  _\ell}\\
\gamma ^\ell_+&=\left(\gamma _-^\ell\gamma _+^\ell \right)^{\frac{1}{2}}\left(\gamma
  _-^\ell/\gamma _+^\ell \right)^{-\frac{1}{2}}=:\rho _\ell e^{-i\theta _\ell},
\end{split}
\end{equation}
where we choose the branches of the square roots close to 1 and the
logarithm close to 0, so that $\rho =1+{\cal O}(h)$, $\theta
={\cal O}(h)$. Then
\begin{equation}\label{ev.7}
\rho ^*=\rho ,
\end{equation}
so $\left(e^{i\theta } \right)^*=e^{-i\theta }$ and hence
\begin{equation}\label{ev.8}
\theta ^*=\theta .
\end{equation}
We write
\begin{equation}\label{ev.9}
g_\ell =\rho _\ell
(e^{\frac{i}{h}\widetilde{I}_\ell}+e^{-\frac{i}{h}\widetilde{I}_\ell}),\
\ \widetilde{I}_\ell =I_\ell +h\theta _\ell =I_\ell +{\cal O}(h^2)=\widetilde{I}_\ell^*. 
\end{equation}

Similarly, we consider
\begin{equation}\label{ev.10}
g=g_r (E,\varepsilon )=\left( e^{\frac{i}{h}I_r}\gamma _+^r
  +e^{-\frac{i}{h}I_r}\gamma _-^r \right) .
\end{equation}
and again we have (\ref{ev.5}), now with $g=g_r$.
Write
\begin{equation}\label{ev.11}
\begin{split}
\gamma ^r_+&=\left(\gamma _-^r\gamma _+^r \right)^{\frac{1}{2}}\left(\gamma
  _+^r/\gamma _-^r \right)^{\frac{1}{2}}=:\rho _r e^{i\theta
  _r}\\
\gamma ^r_-&=\left(\gamma _-^r\gamma _+^r \right)^{\frac{1}{2}}\left(\gamma
  _+^r/\gamma _-^r \right)^{-\frac{1}{2}}=:\rho _r e^{-i\theta _r}.
\end{split}
\end{equation}
again with $\rho =1+{\cal O}(h)$, $\theta
={\cal O}(h)$ satisfying (\ref{ev.7}), (\ref{ev.8}).

We write
\begin{equation}\label{ev.12}
g_r =\rho _r
(e^{\frac{i}{h}\widetilde{I}_r}+e^{-\frac{i}{h}\widetilde{I}_r}),\
\ \widetilde{I}_r= I_r +h\theta _r =I_r +{\cal O}(h^2)=\widetilde{I}_r^*. 
\end{equation}

We now take into account the ${\cal PT}$
symmetry. Clearly, \begin{equation}\label{ev.13} \rho _\ell ^\dagger
  =\rho _r
\end{equation} and from $\left(\gamma _-^\ell
\right)^\dagger = \gamma _-^r$, we get
$\rho _r e^{-i\theta _\ell^\dagger}=\rho _re^{-i\theta _r}$, so  
\begin{equation}\label{ev.14}
\theta _\ell^\dagger = \theta _r \hbox{ and hence }\widetilde{I}_\ell^\dagger=\widetilde{I}_r. 
\end{equation}

\par Using also that $I_\ell ^\dagger =I_r$, we can rewrite $f$ in
(\ref{qc.10}) as 
\[
\begin{split}
  f(E,\varepsilon )=&\frac{\rho \rho ^\dagger}{4}\left(
    e^{\frac{i}{h}\widetilde{I}}+e^{-\frac{i}{h}\widetilde{I}}\right)
  \left( e^{\frac{i}{h}\widetilde{I}^\dagger}+e^{-\frac{i}{h}\widetilde{I}^{\dagger}}\right)\\
  &-\frac{1}{4}e^{-2J/h}\sin (I/h)\sin (I^\dagger /h)\\
=& \rho \rho ^\dagger \cos (\widetilde{I}/h)\cos (\widetilde{I}^\dagger /h)
-\frac{1}{4}e^{-2J/h}\sin (I/h)\sin (I^\dagger /h) 
\end{split}
\]
where $\widetilde{I}=\widetilde{I}_\ell$, $I=I_\ell$, $\rho =\rho
_\ell$. Dividing this function with $\rho \rho ^\dagger$ will not
modify the zeros and we get the new (slightly modified) function that
we shall denote by the same symbol, 
\begin{equation}\label{ev.15}
\begin{split}
  f(E,\varepsilon )= \cos (\widetilde{I}/h)\cos (\widetilde{I}^\dagger /h)
-\frac{1}{4}e^{-2\widetilde{J}/h}\sin (I/h)\sin (I^\dagger /h),
\end{split}
\end{equation}
where 
\begin{equation}\label{ev.15.5}
\widetilde{J}=J+h\ln (\rho \rho ^\dagger )=J+{\cal O}(h^2).
\end{equation}
We know that 
\begin{equation}\label{ev.16}
f^\dagger =f= f^*.
\end{equation}
From 
\begin{equation}\label{ev.17}
I(E,\varepsilon )=\int_{\alpha _\ell}^{\beta _\ell}(E-V_\varepsilon (x))^{\frac{1}{2}}dx,
\end{equation}
we get,
\begin{equation}\label{ev.18}
\partial _E I(E,\varepsilon )=\frac{1}{2}\int_{\alpha _\ell}^{\beta _\ell}(E-V_\varepsilon (x))^{-\frac{1}{2}}dx,
\end{equation}
and
\begin{equation}\label{ev.19}
\partial _\varepsilon  I(E,\varepsilon )=\frac{1}{2i}\int_{\alpha _\ell}^{\beta _\ell}(E-V_\varepsilon (x))^{-\frac{1}{2}}W(x)dx.
\end{equation}
It follows that
\begin{equation}\label{ev.20}\partial _EI(E,0)>0,\ i\partial
  _\varepsilon I(E,0)\in \mathbb{ R}\hbox{ when }E\in
  \mathrm{neigh\,}(E_0,\mathbb{ R}).\end{equation}
We now adopt the assumption (A7) so that the integral in (\ref{ev.19})
is non-vanishing for $(E,\varepsilon )=(E_0,0)$ and in order to fix the
ideas (possibly after replacing $(\varepsilon ,W)$ by $(-\varepsilon ,-W)$) that
\begin{equation}\label{ev.20.5}
\int_{\alpha _\ell}^{\beta _\ell}(E_0-V_0(x))^{-\frac{1}{2}}W(x)dx>0,
\end{equation}
so that
\begin{equation}\label{ev.21}
i\partial _\varepsilon I(E,0)>0\hbox{ for }E\in \mathrm{neigh\,}(E_0,\mathbb{R}).
\end{equation}

\par Let $I_0=I(E_0,0)\in \mathbb{ R}$. In view of the first part of
(\ref{ev.20}), the map $I(\cdot ,\varepsilon
):\,\mathrm{neigh\,}(E_0,\mathbb{ C})\to \mathrm{neigh\,}(I_0,\mathbb{ C})$ is
bijective for $\varepsilon \in \mathrm{neigh\,}(0,\mathbb{ R})$ with an
inverse $K(\cdot ,\varepsilon ):\, \mathrm{neigh\,}(I_0,\mathbb{ C})\to
\mathrm{neigh\,} (E_0,\mathbb{ C})$ such that $K(\iota ,\varepsilon )$ is
holomorphic in $(\iota ,\varepsilon )$. We also know that $K(\iota ,0)$ is real
when $\iota $ is real. The property $\widetilde{I}^*=\widetilde{I}$ implies
that $\widetilde{I}(E,\varepsilon )$ is real when $\varepsilon =0$ and since
$\widetilde{I}=I+{\cal O}(h^2)$, we see that $\widetilde{I}(\cdot
,\varepsilon )$ has a local inverse $\widetilde{K}$ with the same
properties as $K$. Further, $\widetilde{K}=K+{\cal O}(h^2)$ has a
complete asymptotic expansion in powers of $h$ in the space of
holomorphic functions defined in a neighborhood of $(I_0,0)$.

\par The zeros of the factor $\cos (\widetilde{I}/h) $ are given by 
\begin{equation}\label{ev.22}
\widetilde{I}(E,\varepsilon )=\left(k+\frac{1}{2} \right)\pi h, 
\end{equation}
for $k\in \mathbb{ Z}$ such that $\left(k+\frac{1}{2} \right)\pi h$
belongs to a neighborhood of $I_0$. (The classical action for the left
potential well is equal to $2I$, so we recognize the Bohr-Sommerfeld
quantization condition $2\widetilde{I}(E,\varepsilon )=(k+1/2)2\pi h$.) They are
situated on the real-analytic curve
\begin{equation}\label{ev.23}\widetilde{\Gamma }(\varepsilon )=\{ E\in
  \mathrm{neigh\,}(E_0,\mathbb{ C});\, \widetilde{I}(E,\varepsilon )\in \mathbb{R}\}. \end{equation} Alternatively, the zeros are of the form
\begin{equation}\label{ev.24}
\widetilde{E}_k=\widetilde{K}((k+1/2)\pi h,\varepsilon ),
\end{equation} 
and the curve (\ref{ev.23}) is of the form $\widetilde{\Gamma
}(\varepsilon )=\widetilde{K}(\mathrm{neigh\,}(0,\mathbb{ R}),\varepsilon )$. 

\par The curve $\widetilde{\Gamma }$ can also be represented in the
form 
$$
\widetilde{\Gamma }:\ \ \im E=\widetilde{g}(\re E,\varepsilon ),
$$
where 
$$
\widetilde{g}(t,\varepsilon )\sim g(t,\varepsilon )+hg_1(t,\varepsilon )+...,
$$
and 
$$
\Gamma :\ \ \im E=g(\re E,\varepsilon )
$$
is the curve, determined by the condition $I(E,\varepsilon )\in \mathbb{
  R}$. We know that $\widetilde{\Gamma }$, $\Gamma $ are real segments
when $\varepsilon =0$, so 
$$
\widetilde{g},\ g={\cal O}(\varepsilon ). 
$$

\par Writing 
$$
\im I(\re E+ig(\re E,\varepsilon ),\varepsilon )=0,
$$
and differentiating with respect to $\varepsilon $ at $\varepsilon =0$, we get
$$
\partial _\varepsilon g(\re E,0)=\frac{i\partial_\varepsilon  I}{\partial _E I}(\re
E,0)>0.
$$
Hence, by Taylor expansion,
$$
g(\re E,\varepsilon )=\frac{i\partial _\varepsilon I }{\partial _E I}(\re E,0) \varepsilon +{\cal O}(\varepsilon ^2). $$
It follows that
$$
\widetilde{g}(\re E,\varepsilon )=\left(\frac{i\partial _\varepsilon I}{\partial _E I}(\re
E,0) +{\cal O}(h^2) \right) \varepsilon +{\cal O}(\varepsilon ^2). $$

Similarly, by differentiating the equation
$\widetilde{I}(\widetilde{E}_k(\varepsilon ),\varepsilon )=(k+1/2)\pi h$
with respect to $\varepsilon $, we get 
$$
\partial _\varepsilon \widetilde{E}_k(\varepsilon
)=i\frac{i\partial _\varepsilon \widetilde{I}}{\partial
  _E\widetilde{I}}\left(\widetilde{E}_k(\varepsilon ),\varepsilon  \right),
$$
and again by Taylor expansion at $\varepsilon =0$,
\begin{equation}\label{ev.24.5}\begin{split}
\widetilde{E}_k(\varepsilon
)&=\widetilde{E}_k(0)+i\frac{i\partial _\varepsilon \widetilde{I} }{\partial
  _E\widetilde{I}}\left(\widetilde{E}_k(0 ),0 \right)\varepsilon +{\cal
  O}(\varepsilon ^2)\\
&=({E}_k(0)+{\cal O}(h^2))+i\left( \frac{i\partial _\varepsilon
  {I}}{\partial
  _E{I}}\left(\widetilde{E}_k(0 ),0 \right)+{\cal O}(h^2) \right)\varepsilon +{\cal
  O}(\varepsilon ^2)
\end{split}
\end{equation}
Here $\widetilde{E}_k(0)$ and $E_k(0)$ are real and given by the
Bohr-Sommerfeld conditions
\[
\widetilde{I}(\widetilde{E}_k(0),0)=(k+1/2)\pi h,\ \ 
{I}({E}_k(0),0)=(k+1/2)\pi h .
\]

\par The zeros of the second factor
$$
\cos(\widetilde{I}^\dagger /h)(E,\varepsilon )=\overline{\cos
  (\widetilde{I}/h)(\overline{E},\varepsilon )},
$$
are given by $E=\overline{\widetilde{E}_k(\varepsilon
  )}$ and they are situated on the complex conjugate curve
$$
\im E=-\widetilde{g}(\re E,\varepsilon ).
$$

\par We next make an exponential localization of the zeros of $f$. We
have,
$$
|\cos z|\asymp \min (\mathrm{dist\,}(z,(\mathbb{ Z}+1/2)\pi
),1)e^{|\im z|},
$$
so for $E\in \mathrm{neigh\,}(E_0,\mathbb{ C})$,
$$
|\cos (\widetilde{I}/h)|\asymp \min
\left(\frac{1}{h}\mathrm{dist\,}(E,\{\widetilde{E}_k \}),1 \right)
e^{|\im  \widetilde{I}|/h}.
$$
Since $\widetilde{I}=I+{\cal O}(h^2)$, we get
\begin{equation}\label{ev.25}
|\cos (\widetilde{I}/h)|\asymp \min
\left(\frac{1}{h}\mathrm{dist\,}(E,\{\widetilde{E}_k \}),1 \right)
e^{|\im  I|/h}.
\end{equation}
Similarly,
\begin{equation}\label{ev.26}
|\cos (\widetilde{I}^\dagger /h)|\asymp \min
\left(\frac{1}{h}\mathrm{dist\,}(E,\{\overline{\widetilde{E}_k } \}),1 \right)
e^{|\im  I^\dagger |/h},
\end{equation}
\begin{equation}\label{ev.27}
|\sin I/h|\le 2e^{|\im  I|/h},
\end{equation}
and
\begin{equation}\label{ev.28}
|\sin I^\dagger /h|\le 2e^{|\im  I^\dagger |/h}.
\end{equation}
We conclude that $f(E,\varepsilon )\ne 0$ when
$$
|\cos (\widetilde{I}/h)\cos (\widetilde{I}^\dagger /h)|\gg e^{-2\re
  J/h}|\sin (I/h)\sin (I^\dagger /h)|
$$
and that this holds when
\begin{equation}\label{ev.29}
\min \left(\frac{1}{h}\mathrm{dist\,}(E,\{\widetilde{E}_k \} ),1
\right)
\min \left(\frac{1}{h}\mathrm{dist\,}(E,\{\overline{\widetilde{E}_k }
  \} ),1 \right)\gg e^{-2\re J(E,\varepsilon )/h}.
\end{equation}
Thus if $C>0$ is large enough, then $f(E,\varepsilon )\ne 0$ when
\begin{equation}\label{ev.30}
E\not\in \bigcup_k D(\widetilde{E}_k,Che^{-\re J(E,\varepsilon )/h})
\cup    
\bigcup_k D(\overline{\widetilde{E}_k},Che^{-\re J(E,\varepsilon )/h}).
\end{equation}
Now we observe that for any $\widehat{E}\in \mathrm{neigh\,}(E_0,\mathbb{
  C})$,
\[
\begin{split}
  E\in D(\widehat{E},Che^{-\re J(E,\varepsilon )/h})&\Longrightarrow E\in
  D(\widehat{E},2Che^{-\re J(\widehat{E},\varepsilon )/h}),\\
E\in D(\widehat{E},2Che^{-\re J(E,\varepsilon )/h})&\Longleftarrow E\in
  D(\widehat{E},Che^{-\re J(\widehat{E},\varepsilon )/h}).
\end{split}
\]
After doubling the constant in (\ref{ev.30}), we conclude that
\begin{equation}\label{ev.31}
f^{-1}(0,\varepsilon )\subset \bigcup_k D(\widetilde{E}_k,Che^{-\re J(\widetilde{E}_k,\varepsilon )/h})
\cup    
\bigcup_k D(\overline{\widetilde{E}_k },Che^{-\re J(\overline{\widetilde{E}_k},\varepsilon )/h}).
\end{equation}

If $E\in D\left(\widetilde{E}_k,Che^{-\re J(\widetilde{E}_k)/h)}\right)$ and if
\begin{equation}\label{ev.32}
\varepsilon \gg he^{-\re J(\widetilde{E}_k)/h},
\end{equation}
then $\mathrm{dist\,}(E,\{\overline{\widetilde{E}_n} \} )\asymp
\varepsilon $, and from (\ref{ev.29}) we conclude that $f(E,\varepsilon )\ne
0$ if
\begin{equation}\label{ev.33}
\frac{1}{h}\mathrm{dist\,}(E,\widetilde{E}_k)\min \left(\frac{\varepsilon
  }{h},1 \right) \gg e^{-2\re J(\widetilde{E}_k,\varepsilon )/h}.
\end{equation}
Thus, the zeros of $f$ in $D(\widetilde{E}_k,Che^{-\re
  J(\widetilde{E}_k)/h})$ are contained in 
$$
D\left(\widetilde{E}_k,\frac{Ch}{\min (\varepsilon /h,1)}e^{-2\re
    J(\widetilde{E}_k,\varepsilon )/h} \right).
$$
If we drop the assumption (\ref{ev.32}), we have
\begin{equation}\label{ev.34}
\begin{split}
&f^{-1}(0,\varepsilon )\bigcap D\left(\widetilde{E}_k , Che^{-\re
    J(\widetilde{E}_k,\varepsilon )/h} \right)\\
&\subset D\left(\widetilde{E}_k,Ch\min \left( 1, \max (h/\varepsilon
  ,1)e^{-\re J(\widetilde{E}_k,\varepsilon )/h} \right) e^{-\re
J(\widetilde{E}_k,\varepsilon )/h}\right).
\end{split}
\end{equation}
The same discussion is valid with $\widetilde{E}_k$ replaced by
$\overline{\widetilde{E}_k}$ and we get the following improvement of (\ref{ev.31}):
\begin{equation}\label{ev.35}
f^{-1}(0,\varepsilon )\subset \bigcup_k
D\left( \widetilde{E}_k,r(\widetilde{E}_k,\varepsilon )\right)
\cup
\bigcup_{k}D\left(\overline{\widetilde{E}_k},r\left(\overline{\widetilde{E}_k},\varepsilon
  \right) \right),
\end{equation}
where 
\begin{equation}\label{ev.36}
r(E,\varepsilon )=Ch \min \left( 1,\max (h/\varepsilon ,1)e^{-\re
    J(E,\varepsilon )/h} \right) e^{-\re J(E,\varepsilon )/h}.
\end{equation}

\par By inserting a deformation parameter $\theta \in [0,1]$ in front
of the second term in the last expression for $f$ in (\ref{ev.15}), we
will not change the localization (\ref{ev.35}) of the zeros and the
number of such zeros in each connected component of the set in the
right hand side of that inclusion is independent of $\theta $. It
follows that
\begin{itemize}
\item when these discs are disjoint, $f(\cdot ,\varepsilon )$ has precisely one zero in each of
  $D\left(\widetilde{E}_k,r\left(\widetilde{E}_k,\varepsilon \right)\right)$ and $D\left(\overline{\widetilde{E}_k},r\left(\overline{\widetilde{E}_k},\varepsilon
  \right) \right)$.
\item in general $f(\cdot ,\varepsilon )$ has precisely 2 zeros in
  $$D\left(\widetilde{E}_k,r\left(\widetilde{E}_k,\varepsilon
\right)\right)\cup
    D\left(\overline{\widetilde{E}_k},r\left(\overline{\widetilde{E}_k},\varepsilon
      \right) \right).$$
\end{itemize}

\par From the first of these observations, (\ref{ev.31}), (\ref{ev.24.5}) and the fact
that 
\begin{equation}\label{ev.37}
\frac{i\partial _\varepsilon I}{\partial _EI}(E,0)>0\hbox{ when }E\hbox{
  is real,}
\end{equation}
we get
\begin{proposition}\label{ev1}
Assume (\ref{ev.20.5}) for $E\in \mathrm{neigh\,}(E_0,\mathbb{ R})$ and
that $\varepsilon $ is real and 
$$
1\gg |\varepsilon |\ge he^{-(\re J(E_0)-1/C)/h}
$$
for some positive constant $C$. Then the eigenvalues in
$\mathrm{neigh\,}(E_0,\mathbb{ C})$ are simple and non-real of the form
$z_k(\varepsilon ;h)$, $\overline{z_k(\varepsilon ;h)}$, $k\in \mathbb{ Z}$,
where
$$
z_k=\widetilde{E}_k(\varepsilon ;h)+{\cal O}(h)e^{-\re
  J(\widetilde{E}_k,\varepsilon )/h}
$$
and we recall (\ref{ev.24.5}), (\ref{ev.24}). The term ${\cal
  O}(h)e^{-\re J(\widetilde{E}_k,\varepsilon )/h}$ can be replaced by
${\cal O}(r(\widetilde{E}_k,\varepsilon ))$, where $r(E,\varepsilon )$
is defined in (\ref{ev.36}).
\end{proposition}

\par It remains to make a more detailed study, when 
$$
|\varepsilon |\le he^{-(\re J(E_0)-1/C)/h},
$$
and for that we shall view $\widetilde{E}_k(0;h)$ as the nondegenerate
local minima of $f_0(E,0)$, $E\in \mathrm{neigh\,}(E_0,\mathbb{ R})$,
where we put for $0\le \theta \le 1$,
\begin{equation}\label{ev.37.5}
f_\theta (E,\varepsilon )=\cos (\widetilde{I}/h)\cos
(\widetilde{I}^\dagger /h)-\frac{\theta }{4}e^{-2\widetilde{J}/h}\sin (I/h)\sin
(I^\dagger /h),
\end{equation}
so that $f_1=f$ in (\ref{ev.15}). Using that $f_\theta ^*=f_\theta $
and $f_\theta ^\dagger =f_\theta $ we see that for real $E$
\begin{itemize}
\item $f_\theta $ is real-valued,
\item $f_\theta $ is an even function of $\varepsilon $.
\end{itemize}

\par Write 
$$
f_0=g(E,\varepsilon )g^\dagger (E,\varepsilon ),\ \ g(E,\varepsilon )=\cos
(\widetilde{I}/h)=g^*(E,\varepsilon ).
$$
Let $E_c(0)$ be a (real) zero of $g(E,0)$, so that
$E_c(0)=\widetilde{E}_k(0)$ for some $k\in \mathbb{ Z}$. For $\varepsilon
=0$, we have $f_0=g(E,0)^2\ge 0$ and $E_c(0)$ is therefore a nondegenerate
local minimum of $f_0$ with $f_0(E_c(0),0)=0$. Extend $E_c(0 )$ to an
analytic family $E_c(\varepsilon )$ of critical points of $f_0(\cdot
,\varepsilon )$:
\begin{equation}\label{ev.38}
\partial _Ef_0 (E_c(\varepsilon ),\varepsilon )=0. 
\end{equation}

We have for real $E$:
\[
\partial _Ef_0(E,\varepsilon )=2\re (\partial _Eg(E,\varepsilon
)\overline{g(E,\varepsilon )}),
\]
and differentiating this once more and putting $\varepsilon =0$,
$E=E_c(0)$, we get
$$
\partial _E^2f_0(E_c(0),0)=2\partial _Eg(E_c(0),0)\overline{\partial
  _Eg(E_c(0),0)}=2\left(\partial _Eg(E_c(0),0) \right)^2,
$$
i.e.
\begin{equation}\label{ev.39}\begin{split}
h^2\partial _E^2f_0(E_c(0),0)=&2\left(\sin (\widetilde{I}(E_c(0),0)/h)
\right)^2 \left(\partial _E \widetilde{I}(E_c(0),0) \right)^2\\
=& 2 \left(\partial _E \widetilde{I}(E_c(0),0) \right)^2 ,
\end{split}
\end{equation}
where the last identity follows from the fact that $\cos
\widetilde{I}(E_c(0),0)/h =0$.

\par Differentiating (\ref{ev.38}), we get
$$
\left( \partial _E^2 f_0 \right) \partial _\varepsilon E_c+\partial
_\varepsilon \partial _E f_0=0.
$$
Here we recall that when $E$ is real, $f_0(E,\varepsilon )$ and $\partial
_E f_0(E,\varepsilon )$ are even functions of $\varepsilon $ and hence
$\partial _\varepsilon \partial _Ef_0(E,0)=0$.
It follows that 
\begin{eqnarray}\label{ev.40}
&&\left(\partial _\varepsilon \partial _E f_0 \right) (E_c(0),0)=0,
\\[12pt]
\label{ev.41}
&&\partial _\varepsilon E_c(0)=0.
\end{eqnarray}
Using this, we get
\[
\begin{split}
\left(\partial _\varepsilon  \right)^2_{\varepsilon =0}&(f_0(E_c(\varepsilon
),\varepsilon ))
\\
&=\left(\partial _\varepsilon  \right)_{\varepsilon =0}\left(
  \underbrace{\left(\partial _Ef_0 \right) (E_c(\varepsilon ),\varepsilon )}_{=0}
\partial _\varepsilon E_c(\varepsilon )+\left(\partial _\varepsilon f_0
\right) (E_c(\varepsilon ),\varepsilon ) \right)\\
&=\left(\partial _\varepsilon  \right)^2 f_0(E_c(0),0)+\underbrace{
  \left(\partial _E\partial _\varepsilon f_0 \right)(E_c(\varepsilon
  ),\varepsilon )}_{=0}\underbrace{\partial _\varepsilon E_c(0)}_{=0}\\
&= 2\left| \partial _\varepsilon g(E_c(0),0) \right|^2.
\end{split}
\]
Thus,
\begin{equation}\label{ev.42}\begin{split}
h^2\left(\partial _\varepsilon  \right)^2_{\varepsilon =0}\left(
  f_0(E_c(\varepsilon ),\varepsilon ) \right)=&2\left(\sin
  \widetilde{I}(E_c(0),0)/h \right) ^2 \left| \left(\partial _\varepsilon
  \widetilde{I}\right) (E_c(0),0) \right| ^2\\
=&2 \left| \left(\partial _\varepsilon
  \widetilde{I}\right) (E_c(0),0) \right| ^2.
\end{split}
\end{equation}

\par We next extend $E_c(\varepsilon )$ to an analytic function
$E_c(\varepsilon ,\theta )$ determined by the conditions $E_c(\varepsilon
,0)=E_c(\varepsilon )$,
\begin{equation}\label{ev.43}
\partial _Ef_\theta (E_c(\varepsilon ,\theta ),\varepsilon )=0.
\end{equation}

\par It will be convenient to restrict the attention to a window of
size ${\cal O}(h)$:
\begin{equation}\label{ev.44}
E=E_1+hF,\ \ \varepsilon =h\widetilde{\varepsilon },
\end{equation}
where $E_1\in \mathrm{neigh\,}(E_0,\mathbb{R})$ is a parameter and
$F\in \mathrm{neigh\,}(0,\mathbb{C})$, $\widetilde{\varepsilon }\in
\mathrm{neigh\,}(0,\mathbb{R})$ are rescaled variables. It will also
be convenient to have a ``Floquet parameter'' $\kappa \in \mathbb{R}$
and introduce the following extension of (\ref{ev.37.5}):
\begin{equation}\label{ev.45}
\begin{split}
f_\theta &(E,\varepsilon ,\kappa  )=
\\
&\cos \left( \frac{\widetilde{I}}{h}-\kappa \right)\cos
\left( \frac{\widetilde{I}^\dagger}{h}-\kappa \right)-\frac{\theta
}{4}e^{-2\widetilde{J}/h}
\sin \left( \frac{I}{h}-\kappa \right)\sin
\left( \frac{I^\dagger}{h}-\kappa \right),
\end{split}
\end{equation}
which coincides with $f_\theta (E,\varepsilon )$, when $\kappa \in \pi
\mathbb{Z}$. Again, $f_\theta $ is real-valued when $E$ is real and an
even function of $\varepsilon $. If we let $E_c(\varepsilon ,\kappa ,\theta
)$ denote a local minimum of $f_\theta (\cdot ,\varepsilon ,\kappa )$,
then (\ref{ev.39}), (\ref{ev.40}), (\ref{ev.41}) extend naturally to
the case $\theta =0$. Writing 
$$
\kappa =\widetilde{\kappa }+I(E_1,0)/h,
$$
we get 
\begin{equation}\label{ev.46}
f_\theta = \widetilde{f}_\theta (F,\widetilde{\varepsilon
},\widetilde{\kappa };h)=a(F,\widetilde{\varepsilon },\widetilde{\kappa };h)+
\theta e^{-2J(E_1,0)/h}b(F,\widetilde{\varepsilon },\widetilde{\kappa };h),
\end{equation}
where $a$, $b$ are classical symbols of order $0$ in $h$:
$$
a\sim a_0+ha_1+...,\ \ b\sim b_0+hb_1+...
$$

The critical points with respect to $F$ are nondegenerate and their
number is uniformly bounded. They have asymptotic expansions in powers
of $h$ of the form,
$$F_c(\widetilde{\varepsilon },\widetilde{\kappa },\theta ;h)\sim 
F_c^0(\widetilde{\varepsilon },\widetilde{\kappa },\theta e^{-2J(E_1,0)/h)})+h F_c^1(\widetilde{\varepsilon },\widetilde{\kappa },\theta e^{-2J(E_1,0)/h)})+...
$$ which gives
\begin{equation}\label{ev.47}
F_c(\widetilde{\varepsilon },\widetilde{\kappa },\theta
;h)=F_c^1(\widetilde{\varepsilon },\widetilde{\kappa } ;h)+\theta
e^{-2J(E_1,0)/h}
F_c^2(\widetilde{\varepsilon },\widetilde{\kappa },\theta ;h),
\end{equation}
where $F_c^k$ are classical symbols of order $0$ in $h$ and also
holomorphic functions. 
 Notice that the terms in the asymptotic expansion of $F_c^2$
in powers of $h$ are independent of $\theta $.  
$E_c=E_1+hF_c$ will be a critical point of $f_\theta (\cdot ,\varepsilon
)$ in (\ref{ev.37.5}) when $\kappa \in \pi \mathbb{Z}$, i.e. when 
\begin{equation}\label{ev.48}
\widetilde{\kappa }\equiv -\frac{I(E_1,0)}{h}\ \mathrm{mod}\ \pi \mathbb{Z}.
\end{equation}

\par It follows from (\ref{ev.47}) that 
$$
F_c(\widetilde{\varepsilon },\widetilde{\kappa },\theta
;h)=F_c(\widetilde{\varepsilon },\widetilde{\kappa },0;h)+{\cal
  O}(1)\theta e^{-2J(E_1,0)/h},
$$
and hence that 
\begin{equation}\label{ev.49}
E_c(\varepsilon ,\kappa ,\theta )=E_c(\varepsilon ,\kappa ,0)+{\cal O}(h)\theta e^{-2J(E_1,0)/h}.
\end{equation}
To evaluate the critical value $f_\theta (E_c(\varepsilon ,\kappa ,\theta
),\varepsilon ,\kappa )=:f_\theta ^c(\varepsilon ,\kappa )$ for $\theta =1$, we notice that 
\[
\begin{split}
\partial _\theta & (f_\theta (\varepsilon ,\kappa ))=(\partial _\theta
f_\theta )(E_c(\varepsilon ,\kappa ,\theta ))\\
=&-\frac{1}{4}\left( e^{-2\widetilde{J}/h}\sin \left(\frac{I}{h}-\kappa  \right)
\sin \left(\frac{I^\dagger}{h}-\kappa  \right) \right) (E_c(\varepsilon ,\kappa
,\theta ),\varepsilon ,\kappa )\\
=&-\frac{1}{4} e^{-2\widetilde{J}(E_c(0,\kappa ,0),0)} \times \\ 
&\left[\left(\sin
    \left(\frac{I}{h}-\kappa  \right) \right)^2(E_c(0,\kappa
  ,0),0,\kappa ))+{\cal O}(1)\left( \theta
    e^{-2\re J(E_1,0)/h}+\frac{|\varepsilon |}{h} \right) \right].
\end{split}
\]
Here, we use that 
\[
\sin \left(\frac{\widetilde{I}}{h}-\kappa  \right) (E_c(0,\kappa
,0),0,\kappa )=\pm 1,\ \ \sin\left(\frac{I}{h}-\kappa  \right)=\sin\left(\frac{\widetilde{I}}{h}-\kappa  \right)+h
\]
and integrate from $\theta =0$ to $\theta =1$, to get
\begin{equation}\label{ev.50}\begin{split}
&f_1^c(\varepsilon ,\kappa )=f_0^c(\varepsilon ,\kappa )\\
&-\frac{1}{4}e^{-2J(E_c(0,\kappa
  ,0),0)/h}\left(1+{\cal O}\left(e^{-2\re
      J((E_1,0),0)/h}+\frac{|\varepsilon |}{h}+{\cal O}(h) \right) \right).
\end{split}
\end{equation}

\par Now, return to the window (\ref{ev.44}), where $f_\theta
=\widetilde{f}_\theta (F,\widetilde{\varepsilon },\widetilde{\kappa };h)$
is given by (\ref{ev.46}) and the critical point $F_c$ is as in
(\ref{ev.47}). We have with $\theta =1$ (and suppressing the
corresponding subscript 1)  
\begin{equation}\label{ev.51}
f^c(\varepsilon ,\kappa ;h)=\widetilde{f}^c(\widetilde{\varepsilon
},\widetilde{\kappa };h)=g_1(\widetilde{\varepsilon },\widetilde{\kappa
};h)+g_2(\widetilde{\varepsilon },\widetilde{\kappa
};h)e^{-2J(E_c(0,\kappa ,0),0)/h},
\end{equation}
where $g_j$ are classical symbols of order $0$ in $h$. (We first get
this with $J(E_1,0)/h$ in the exponent, but the replacement by
$J(E_c(0,\kappa ,0),0)$ does not modify the general structure of the
formula.) (\ref{ev.50}) shows that
\begin{equation}\label{ev.52}
g_1(0,\widetilde{\kappa };h)=0,\ \ g_2(0,\widetilde{\kappa })=-\frac{1}{4},
\end{equation} 
where $g_{j,0}$ is the leading term in the asymptotic expansion of $g_j$.

From (\ref{ev.42}), we deduce that
\begin{equation}\label{ev.53}
\partial _{\widetilde{\varepsilon }}^2g_{1,0}(0,\widetilde{\kappa })=2
|\partial _\varepsilon I(E_1,0)|^2>0.
\end{equation}

\par Combining (\ref{ev.51}), (\ref{ev.52}), (\ref{ev.53}), we get by
Taylor expansion,
\begin{equation}\label{ev.54}
\widetilde{f}^c(\widetilde{\varepsilon },\widetilde{\kappa
};h)=g_2(0,\widetilde{\kappa };h)e^{-2J(E_c(0,\kappa
  ,0),0)/h}+k(\widetilde{\varepsilon },\widetilde{\kappa
};h)\widetilde{\varepsilon }^2,
\end{equation}
where $k$ is a symbol of order $0$ in $h$, holomorphic in the other
variables, even in $\widetilde{\varepsilon }$ and satisfying
$k(0,\widetilde{\kappa };0)=|\partial _\varepsilon
I(E_1,0)|^2$. $\widetilde{f}^c(\cdot ,\kappa ;h)$ has precisely two
zeros in a neighborhood $0$ which are real and of the form $\pm
\widetilde{\varepsilon }_c(\kappa ;h)$, where
\begin{equation}\label{ev.55}
\widetilde{\varepsilon }_c(\kappa ;h)=\ell (\widetilde{\kappa
};h)e^{-J(E_c(0,\kappa ,0),0)/h},
\end{equation}
and $\ell$ is a symbol of order $0$ with leading term
$$
\ell_0(\widetilde{\kappa })=\frac{1}{2|\partial _\varepsilon I(E_1,0)|}\cdotp
$$
Using that our functions are holomorphic in $\widetilde{\varepsilon }$,
$\widetilde{\kappa }$, we see that 
\begin{equation}\label{ev.56}
\widetilde{f}^c(\widetilde{\varepsilon },\widetilde{\kappa } ;h)=m(\widetilde{\varepsilon
},\widetilde{\kappa };h)(\widetilde{\varepsilon }^2-\widetilde{\varepsilon
}_c(\widetilde{\kappa };h)^2),
\end{equation}
where $m$ is holomorphic in $\widetilde{\varepsilon }$,
$\widetilde{\kappa }$ and a symbol of order $0$ in $h$ with leading
term $m_0(\widetilde{\varepsilon },\widetilde{\kappa })$, satisfying 
\begin{equation}\label{ev.57}
m_0(0,\widetilde{\kappa })=|\partial _\varepsilon I(E_1,0)|.
\end{equation}

\par (\ref{ev.39}) can be extended to the $\kappa $-dependent case:
\begin{equation}\label{ev.58}
(h^2\partial _E^2 f_0)(E_c(0,\kappa ,0),\kappa ,0)=2(\partial
_E\widetilde{I}(E_c(0,\kappa ,0))^2.
\end{equation}
This implies that
\begin{equation}\label{ev.59}
(\partial _F^2\widetilde{f}_1)(F_c,\widetilde{\varepsilon
},\widetilde{\kappa };h)=(\partial _EI(E_c(0,\kappa ,0),0))^2+{\cal O}(h),
\end{equation}
where the remainder has a complete asymptotic expansion in powers of
$h$.
By Taylor expansion,
\begin{equation}\label{ev.60}
\widetilde{f}_1(F,\widetilde{\varepsilon },\widetilde{\kappa
};h)=\widetilde{f}^c(\widetilde{\varepsilon }, \widetilde{\kappa
};h)+q(F,\widetilde{\varepsilon } ,\widetilde{\kappa };h)(F-F_c(\widetilde{\varepsilon },\widetilde{\kappa },1;h))^2,
\end{equation}
where $q>0$ is a symbol of order $0$ and 
\begin{equation}\label{ev.61}
q(F_c,0,\widetilde{\kappa };0)=2(\partial _E\widetilde{I}(E_c(0,\kappa ,0),0))^2.
\end{equation}

\par $\widetilde{f}(\cdot ,\widetilde{\varepsilon },\widetilde{\kappa
};h)$ has two zeros in a small neighborhood of $F_c$ when counted with
their multiplicity:
\begin{itemize}
\item When $|\widetilde{\varepsilon }|<\widetilde{\varepsilon
  }_c(\widetilde{\kappa };h)$ the zeros are real and simple, given by
\begin{equation}\label{ev.62}
q(F,\widetilde{\varepsilon },\widetilde{\kappa
};h)^{\frac{1}{2}}(F-F_c(\widetilde{\varepsilon },\widetilde{\kappa
},1;h))=\pm (-\widetilde{f}^c(\widetilde{\varepsilon },\widetilde{\kappa
};h))^{\frac{1}{2}}
\end{equation}
\item When $|\widetilde{\varepsilon }|=\widetilde{\varepsilon
  }_c(\widetilde{\kappa };h)$ we have a double zero, 
\begin{equation}\label{ev.63}
F=F_c.
\end{equation}
\item When $|\widetilde{\varepsilon }|>\widetilde{\varepsilon
  }_c(\widetilde{\kappa };h)$ the zeros are non-real and simple and
  complex conjugate to each other, given by
\begin{equation}\label{ev.64}
q(F,\widetilde{\varepsilon },\widetilde{\kappa
};h)^{\frac{1}{2}}(F-F_c(\widetilde{\varepsilon },\widetilde{\kappa
},1;h))=\pm i (\widetilde{f}^c(\widetilde{\varepsilon},\widetilde{\kappa
};h))^{\frac{1}{2}}.
\end{equation}
\end{itemize}

\bibliographystyle{amsplain}

\end{document}